\documentclass[11pt]{amsart}
\usepackage[utf8]{inputenc}
\usepackage{amsmath}
\usepackage{amsthm}
\usepackage{amssymb,amsfonts}
\usepackage{mathtools}
\usepackage{enumerate}
\usepackage{mathrsfs}
\usepackage{bm}
\usepackage{graphicx}
\usepackage{graphics,color}
\newcommand{\revcol}{\relax}
\newcommand{\Aop}{{\revcol A}}
\newtheorem{theorem}{Theorem}[section]
\newtheorem{prop}[theorem]{Proposition}% [Revision 4] shared counter, to avoid a Lemma/Proposition number clash
\newtheorem{corollary}{Corollary}[theorem]
\newtheorem{lemma}[theorem]{Lemma}

\theoremstyle{remark}
\newtheorem*{remark}{Remark}

\theoremstyle{definition}
\newtheorem{definition}{Definition}[section]
\newcommand{\R}{\mathbb{R}}
\newcommand{\N}{\mathbb{N}}

\newcommand{\Z}{\mathbb{Z}}
\newcommand{\E}{\mathbb{E}}

\newcommand{\abs}[1]{\left\lvert#1\right\rvert}
\newcommand{\norm}[1]{\left\lVert#1\right\rVert}
\newcommand{\snorm}[1]{\lVert #1 \rVert}
\newcommand{\setmid}{\mathrel{} \middle| \mathrel{}}

\newcommand{\iprod}[1]{\left\langle #1 \right\rangle}

\newcommand{\black}{\color{black}}

\begin{document}

\title[Ergodicity for a Constantin-Lax-Majda-DeGregorio model of turbulent flow]{Ergodicity for a Constantin-Lax-Majda-DeGregorio model of turbulent flow at large viscosity}
\author[Shunsuke Fujita, Reika Fukuizumi, Takashi Sakajo]{
    Shunsuke Fujita\textsuperscript{1},
    Reika Fukuizumi\textsuperscript{2},
    Takashi Sakajo \textsuperscript{3}
}
\maketitle
\begin{center} \small
$^{1,2}$ Department of Mathematics, \\
School of Fundamental Science and Engineering, Waseda University,\\
Tokyo 169-0072, Japan; \\
\email{fujimath@fuji.waseda.jp} 
%$^2$ Department of Mathematics, \\
%School of Fundamental Science and Engineering, Waseda University,\\
%Tokyo 169-0072, Japan; \\
$^3$ Department of Mathematics, \\ 
Kyoto University, Kitashirakawa Oiwake-cho,\\
Sakyo-ku, Kyoto 606-8502, Japan
\end{center}

\begin{abstract}
This paper presents a mathematical analysis of a stochastic generalized Constantin-Lax-Majda-DeGregorio (gCLMG) equation. 
We focus on the specific case where the nonlinearity in the equation allows an inviscid conserved quantity, i.e., the squared norm of model vorticity. 
The existence of an invariant measure is proved via the classical Krylov-Bogoliubov argument. 
The uniqueness of the measure and exponential mixing are proved under a sufficiently large viscosity condition, in which the nonlocal structure of the nonlinear term becomes manifest.
\end{abstract}

\section{Introduction}
In this paper, as a one-dimensional turbulence model, we consider the generalized Constantin-Lax-Majda-DeGregorio (gCLMG) equation with a stochastic force on the one-dimensional torus $\mathbb{S}^1=\mathbb{R}/(2\pi \mathbb{Z})$:
\begin{equation}
\omega_{t}+au\omega_x-u_x\omega=\nu \omega_{xx} + f, \qquad u_{x}=\mathcal{H}(\omega), \qquad \omega(0,x)=\omega_0(x), 
\label{gCLMG_eq}
\end{equation}
in which $a$ is a real parameter and $\mathcal{H}$ is the Hilbert transform. The Hilbert transform is defined through Fourier series. If  
\begin{eqnarray*}
\omega(x) =\sum_{k\in \Z} \mathscr{F}{(\omega)}(k) e^{ikx}, \quad  \omega \in C^{\infty}(\mathbb{S}^1), 
\end{eqnarray*}
 then  
\begin{align*}
\mathcal{H}(\omega) = -\mathscr{F}^{-1} i \mathrm{sgn}(k) \mathscr{F} (\omega).\quad 
\end{align*}
Here, $\mathrm{sgn}(k) = k/|k|$ for $k\neq 0$ and $0$ for $k=0$.
In \cite{Okamoto_Sakajo_Wunsch_2008}, the authors studied the case \(\nu = f = 0\) which corresponds to the following equation:
\begin{equation}
\omega_{t}+au\omega_x-u_x\omega=0,\qquad u_{x}=\mathcal{H}(\omega), \qquad \omega(0,x)=\omega_0(x). \label{gCLMG-inv}
\end{equation}
This is an evolution equation considered for modeling the vorticity by 
$\omega(t,x)$. 
Apparently, equation \eqref{gCLMG-inv} has two nonlinear terms corresponding to the advection term and the vortex stretching term.

Recall that the following three-dimensional vorticity equation, which is derived from the Navier--Stokes equations for an incompressible fluid, is commonly used to study turbulence. 
\[
\partial_t \boldsymbol{\omega} + (\boldsymbol{u}\cdot\nabla)\boldsymbol{\omega} - (\boldsymbol{\omega}\cdot\nabla)\boldsymbol{u} = \nu\Delta\boldsymbol{\omega} + \boldsymbol{f},
\]
where $\boldsymbol{u}$ is the velocity field and $\boldsymbol{\omega} = \nabla \times \boldsymbol{u}$ is the vorticity.
However, the existence of a global-in-time solution to the three-dimensional Navier--Stokes equations remains a difficult mathematical problem.
There are many more challenges to be overcome in studying turbulence.  
%using this dynamical systems theory approach.
Hence, to better understand this mathematical difficulty, various simpler turbulence models have been proposed.
Such model equations include the Burgers equation~\cite{BEC2007}, and the shell model~\cite{Biferale2003}, moreover, 
%by adding random forcing terms or making the initial conditions random. 
the Constantin-Lax-Majda equation~\cite{Constantin_Lax_Majda_1985}, i.e., 
the equation \eqref{gCLMG-inv} with $a=0$. 
The quadratic term $u_x\omega=\mathcal{H}(\omega)\omega$ in \eqref{gCLMG-inv}  was originally introduced by Constantin, Lax, and Majda~\cite{Constantin_Lax_Majda_1985} as a one-dimensional counterpart to the stretching term $({\bm \omega}\cdot \nabla){\bm u}$ of the vorticity equation.
They constructed an exact analytic solution for the case $a = 0$ in \eqref{gCLMG-inv} that blows up in finite time, revealing that the vortex stretching term is a strong nonlinearity that plays a crucial role in the existence of the solution.
Later, DeGregorio~\cite{DeGregorio_1990} added the advection term, i.e., $a=1$ in \eqref{gCLMG-inv}, which is more relevant to the vorticity equation, and numerically conjectured the existence of global solutions. Okamoto et al.~\cite{Okamoto_Sakajo_Wunsch_2008} have reconsidered the balance between the advection term and the stretching term on the existence of global-in-time solutions in a more general framework by introducing a real parameter $a$.
They have shown the existence of a solution locally in time, belonging to {\revcol $C^0([0,T];\dot{H}^1)\cap C^1([0,T];H)$, in the notation of Section 2,\black} and provided a sufficient condition for the existence of a global solution of the Beale-Kato-Majda type.
Using the condition, the existence of a critical value $0<a_{c}<1$ is numerically conjectured to be such that the solution exists globally in time for $a>a_{c}$.
However, a rigorous mathematical proof remains an open problem.
Moreover, in~\cite{Okamoto_Sakajo_Wunsch_2008}, it has been found that the $L^p$-norm of the solution is a conserved quantity with $p=-a$ for $a\le -1$.
Although the advection term with a negative $a$ has no clear physical relevance, it acquires a preferable mathematical structure having an inviscid conserved quantity.
Therefore, it is expected that this equation can be used as a turbulent model, where the inviscid conserved quantity dissipates anomalously by adding a viscous dissipation term and an external force and this leads us to consider \eqref{gCLMG_eq}. 
When $a=-2$ in particular, the squared norm of the model vorticity corresponding to enstrophy becomes an inviscid conserved quantity. This could be relevant to 2D turbulence, which has already been discussed in \cite{tsuji2023statistical}. 
In \cite{tsuji2023statistical}, the authors consider the case where $f$ in \eqref{gCLMG_eq} is a smooth deterministic force. 
%For this case, they prove global well-posedness and %obtain uniform estimates for the Sobolev norms of %solutions.

We may see \eqref{gCLMG_eq} as a variant of the Burgers equation.
Attempts have been made to understand the turbulent state through the construction of an invariant measure for the stochastic Burgers equation subject to a random external force.
See the reference list in Da Prato-Zabczyk~\cite{da_prato_zabczyk_1992}, in which the existence and uniqueness of a spatially uniform invariant measure have been discussed.
See also Boritchev-Kuksin~\cite{Boritchev_Kuksin_2021_onedimensional} for related results.
The aim of this paper is to construct an invariant measure associated with solutions of the stochastic gCLMG equation and to establish its uniqueness and convergence properties.
For this purpose, we adopt a simplified setting—specifically, a smooth stochastic forcing and a sufficiently large viscosity—so as to clearly distinguish the model from the classical Burgers equation.

{\revcol After having completed this work, we were told about the existence of the paper by Barrera and T\"olle~\cite{Barrera_Tolle} for the ergodicity of locally monotone stochastic evolution equations, in the variational sense of Liu and R\"ockner~\cite{liu_rockner_2015}, with degenerate additive noise. At the level of the $L^2$ space, the stochastic gCLMG equation with $a=-2$ does fit into the framework of \cite{Barrera_Tolle}. The uniqueness of the invariant measure on $L^2$ and the exponential mixing, which is the content of our Section~4 as far as statements on $L^2$ are concerned, also follow from the abstract results of that work under a large-viscosity condition of the same form, $\nu^3\gtrsim B_0$, as ours. On the other hand, the results of the present paper on $\dot{H}^m$ $(m\in\mathbb{N})$  are not consequences of that framework. These are the existence of an invariant measure $\mu$ for \textit{every} $\nu>0$ and the regularity of the invariant measure, $\mu(\dot{H}^{m+1})=1$. Indeed, the cancellation $\iprod{u_x\omega+2u\omega_x,\omega}_{L^2}=0$ is specific to the $L^2$ inner product, and we prove that the coercivity hypotheses of that framework\cite{Barrera_Tolle} fail on the Gelfand triples $\dot{H}^{m+1}\subset\dot{H}^m\subset\dot{H}^{m-1}$ for every $m\ge1$.}
\vspace{3mm}

Section 2 introduces the necessary notation and establishes the global-in-time well-posedness of the equation. Our analysis in this section follows closely the methodology presented in  \cite{Boritchev_Kuksin_2021_onedimensional}, as no essential differences from the Burgers equation arise at this stage. The energy method plays a central role here; in particular, the choice 
$a=-2$ is crucial for the energy estimates and thus proving the global existence of the solution.
Using the uniform energy bounds derived in Section 2, we prove the existence of an invariant measure in Section 3 via the classical Krylov--Bogoliubov argument. Section 4 is devoted to establishing the uniqueness of the invariant measure and to proving the mixing property for sufficiently large viscosity coefficients.
In this part of the analysis, techniques reminiscent of the energy-based global approach developed for the two-dimensional Navier--Stokes equations \cite{mattingly1999ergodicity, constantin2013ergodicity} become indispensable—for instance, in estimating the moments of the invariant measure—whereas in the Burgers case pointwise arguments are often sufficient. We also obtain exponential mixing with respect to the $L^2$-based Lipschitz--dual metric, again in analogy with the Navier--Stokes setting.
{\revcol At the end of Section~4 we record the conclusions that the framework of \cite{Barrera_Tolle} yields on $L^2$, together with an explicit bound for the mixing time in the Wasserstein-2 distance which we derive from the estimates of Section~4. Appendix~C contains the verification of the hypotheses of \cite{Barrera_Tolle} and shows why that framework does not apply beyond $L^2$.}
The ergodic theory for small viscosity remains a highly interesting problem and will require more sophisticated techniques; this will be addressed in subsequent work.

\section{Well-posedness of stochastic gCLMG equation}
%We use the Hilbert transform $\mathcal{H}$ in the formulation of the equation. The definition and properties of the Hilbert transform are given in the appendix. 
To mention our main results precisely, we first explain the notation. 
We begin by defining the following real Hilbert space $H$ as the space of functions with zero mean in $L^2(\mathbb{S}^1)$, on $\mathbb{S}^1 \coloneqq \mathbb{R}/2\pi\mathbb{Z}$.
\[
    H \coloneqq \left\{u \in L^2(\mathbb{S}^1) \setmid  \int_0^{2\pi} u(x) dx = 0 \right\}.
\]
The inner product on $H$ is the standard inner product on $L^2$: $\langle u,v \rangle_{H} \coloneqq \int_0^{2\pi} u(x){v}(x) dx$.

We endow this space $H$ with the real-valued orthonormal basis $\{e_k\}_{k \in \mathbb{Z}^*}$, where $\mathbb{Z}^* \coloneqq \mathbb{Z}\setminus\{0\}$, defined as
\[
    e_k(x) = 
    \begin{cases}
        \frac{1}{\sqrt{\pi}}\cos(kx), & (k \ge 1), \\
        \frac{1}{\sqrt{\pi}}\sin(-kx), & (k \le -1).
    \end{cases}
\]
Any function $u \in H$ can be expanded with Fourier coefficients $u_k$ as
\[
    u(x) = \sum_{k \in \mathbb{Z}^*} u_k e_k(x), \quad \text{where} \quad u_k \coloneqq (\mathscr{F}u)(k)\coloneqq \langle u, e_k \rangle_{H}.
\]
For $m \in \N$,\black we define the Sobolev space $\dot{H}^m$ as the subset of $H$ given by
\begin{equation*}
	\dot{H}^{m} \coloneqq \left\{u \in H \setmid \|u\|_{\dot{H}^m}^2 < \infty \right\},
\end{equation*}
where the norm $\|\cdot\|_{\dot{H}^m}$ is induced by the inner product
\[
    \langle u, v \rangle_{\dot{H}^m} \coloneqq \sum_{k\in \mathbb{Z}^*} |k|^{2m} u_k v_k.
\]
 That is, $\|u\|_{\dot{H}^{m}}^2 = \sum_{k\in\mathbb{Z}^*}|k|^{2m} u_k^2$.
 We write the space of probability measures on $\dot{H}^m$ as $\mathcal{P}(\dot{H}^m)$. 
 By $B_b(\dot{H}^m)$, we denote the set of bounded Borel functions on $\dot{H}^m$. For $\mu\in\mathcal{P}(\dot{H}^m)$ 
 and $f \in B_b(\dot{H}^m)$, we use the notation 
 \[
 \iprod{f, \mu} =\int_{\dot{H}^m} f\;  d\mu
.\]
For $0< T < \infty$, \black
 we then introduce the function space $\dot{X}_{T}^{m} \coloneqq \{f \in X_{T}^{m} \mid f(0) = 0\}$, 
 where $X_T^m\coloneqq C([0,T]; \dot{H}^m)$ denotes the space of continuous functions from $[0, T]$ to $\dot{H}^m$ with the norm
 $\norm{f}_{X_T^m}:=\sup_{0 \leq t \leq T} \norm{f(t)}_{\dot{H}^m}$. 
  We also use the notation $X^m_{\infty}\coloneqq C([0,\infty); \dot{H}^m)$. \black Note that $X_T^m$ is a Banach space and $\dot{X}_T^m$
 is a closed subset of $X_T^m$. 
We denote by $C(a,b,c,\ldots)$, a constant depending on the parameters $a,b,c,\ldots$. For simplicity, the same letter $C$ will be used for various constants in the estimates, as long as no confusion occurs. \black

In equation \eqref{gCLMG_eq}, the Hilbert transform $\mathcal{H}$ for $\omega$ $\in H$ and $k\in \mathbb{Z}$ is defined by
$$
     (\mathscr{F}\mathcal{H}(\omega))(k) = -i \, \mathrm{sgn}(k)\,(\mathscr{F}\omega)(k),
$$
where $\mathrm{sgn}: \R\to\R$ is the sign function.
{\revcol 
The operator $A\coloneqq-\partial_x^2$, with $D(A)=\dot{H}^{m+2}$, is regarded as an unbounded operator on $\dot{H}^m$. It is self-adjoint and strictly positive on the mean-zero space, with discrete spectrum $\{k^2 : k\in\mathbb{Z}^*\}\subset[1,\infty)$. In particular, $A$ is invertible, and the fractional powers $A^{\gamma}$, $\gamma\in\mathbb{R}$, are defined by $(\mathscr{F}A^{\gamma}u)(k)=|k|^{2\gamma}(\mathscr{F}u)(k)$; note that $\norm{f}_{\dot{H}^{m}}=\snorm{A^{m/2}f}_{H}$ and that the velocity in \eqref{gCLMG_eq} is $u=-\Aop^{-1/2}\omega$. Consequently, $-\nu A$ generates on $\dot{H}^m$, for every $m$, an analytic semigroup of contractions $e^{-\nu t\Aop}$, $t\ge0$.}

The random forcing $f(t,x)$ we consider  in \eqref{gCLMG_eq} \black is in the form of
\begin{equation}
f(t, x)= \partial_t \xi(t,x) 
\label{random-forcing}
\end{equation}
with 
\[
\xi(t,x) \coloneqq \sum_{k \in \mathbb{Z}^\ast} b_k \beta_k(t)e_k(x),
\]
 where $\{b_k\}_{k\in \mathbb{Z}^\ast}$ is a sequence of real numbers satisfying $B_{0} \coloneqq \sum_{k\in \mathbb{Z}^\ast} b_k^2<\infty$
and $\{\beta_k(t)\}_{k\in \mathbb{Z}^\ast}$ is a sequence of independent Brownian motions in a filtered probability space 
$(\Omega,\mathcal{F},\mathbb{P}, (\mathcal{F}_t)_{t\ge 0})$, that 
is, $\beta_k(t)$ is $ (\mathcal{F}_t)_{t\ge 0}$ adapted and $\beta_k(t)-\beta_k(s)$ is independent of $\mathcal{F}_{s}$ for $s \le t.$ 
%are independent standard Wiener processes. The Wiener processes $\%{\beta_k(t)\}_{k\in \mathbb{Z}^\ast}$ are independent of the initial %data $\omega_0$ in (\ref{gCLMG_eq}) which are random variables on a %probability space $(\Omega,\mathcal{F},\mathbb{P})$. 
The expectation with respect to $\mathbb{P}$ is denoted by $\mathbb{E}$ in what follows. %$\xi(t)$ is adapted to the filtration $\mathcal{F}_t$.
Note that if the sequence $\{b_k\}_{k\in \Z^*}$ satisfies 
\[
    B_m\coloneqq \sum_{k\in \Z^\ast} \vert k \vert^{2m} b_k^2 <\infty ,
\]
then $\xi\in X_{\infty}^m$ almost surely.
%Moreover, from the independent property, we may assume that $\Omega=\Omega_1 \times \Omega_2$ with $\eta=(\eta_1, \eta_2)$, $\mathcal{F}=\mathcal{F}_1 \otimes \mathcal{F}_2$, and $\mathbb{P}=\mathbb{P}_1 \otimes \mathbb{P}_2$. Then, $\omega_0$ depends only on $\eta_1$. The random forcing $\xi$ is on $\Omega_2$ and $\mathcal{\tilde{F}}_t (\subset \mathcal{F}_2) $-adapted.    
%which we refer to as "degenerate noise" or "colored noise". 
%We now clarify the notion of solution to the equation \eqref{gCLMG_eq}
%subject to the random forcing \eqref{random-forcing} as follows.
%\begin{remark}
%    Unless otherwise specified, we assume that $\omega_0$ is $\mathcal{F}_1\otimes\mathcal{\tilde{F}}_0$-measurable and independent of $\xi$.
%\end{remark}

\begin{definition}\label{def:solution}
Let $T>0$, $a\in \R$ and $\omega_0 \in {\dot{H}^m}$. We say that
{ $\omega \in X_{T}^{m}$}
is a mild solution to the initial value problem,
\begin{equation}
\omega_{t} + au\omega_{x} - u_{x} \omega  - \nu \omega_{xx} = \xi_{t}, \quad u_{x} = \mathcal{H}(\omega), \quad \omega(0) = \omega_{0},
\label{gCLMG_eq_r}
\end{equation}
if 
\begin{equation*}
    \omega(t) = e^{-\nu t\Aop}\omega_0 + \int_0^t e^{-\nu(t-s)\Aop}\{-a(u\omega)_x(s) +(1+a)(u_x\omega)(s)\}ds + \int_0^te^{-\nu(t-s)\Aop}d\xi(s)
\end{equation*} holds in $\dot{H}^m$ for $t\in [0, T]$ and $\mathbb{P}$-$a.s.$
\end{definition}

The basic idea in the proof of the existence of a global solution is the same as that in Boritchev and Kuksin \cite{Boritchev_Kuksin_2021_onedimensional}.
As a standard technique for treating additive noise, we divide equation (\ref{gCLMG_eq_r}) into the linear part and the nonlinear part. In the following lemma, we show the existence of a solution to the stochastic heat equation.

\begin{lemma} \label{S_heat_wellposedness}
Let $0<T<\infty$, $\nu >0$ and  $m \in \N$ \black be fixed. 
Let $m_\ast \in \N$ be such that $m_\ast> m$ \black and assume that $B_{m_*}<\infty$.
Then, there exists a unique $(\mathcal{F}_t)_{t\ge 0}$-adapted solution $v \in X_T^m$ a.s. to the following linear equation 
\begin{equation}\label{S_heat_eq}
v_{t} - \nu v_{xx} = \xi_{t}, \qquad v(0) = 0,
\end{equation}
%Namely,  for all $t\in [0,T]$, and a.s.$\omega\in\Omega$  the following equality holds in $\dot{H}^{m-2}$.
%\begin{equation*}
	%v(t) - \nu \int_{0}^{t} v_{xx}(s) ds = \omega_{0} + \xi(t).
%\end{equation*}
which is represented by
\begin{equation}\label{S_conv}
    v(t) = \int_0^te^{-\nu(t-s)\Aop}d\xi(s). 
\end{equation}
%Furthermore, the mapping  $(\omega_0, \xi) \in H^{m} \times X_{T}^{m_{*}} \mapsto v \in X_{T}^{m}$ is linear and continuous.
\end{lemma}

\begin{proof}
{\revcol
The existence and uniqueness of the solution to \eqref{S_heat_eq}, as well as the representation \eqref{S_conv}, are classical; see \cite[Chapter~5]{DaPrato_Zabczyk_2014}. The pathwise continuity in $\dot{H}^m$ follows from the Kolmogorov criterion \cite[Proposition~3.16]{DaPrato_Zabczyk_2014} applied to the Gaussian process $v$, once the second-moment estimate
\[
\E\bigl[\norm{v(t)-v(s)}_{\dot{H}^m}^2\bigr]\le C\,(t-s)^{\gamma},\qquad t,s\in[0,T],
\]
is established for some $\gamma\in(0,1]$. Since the verification of this estimate uses only the Fourier expansion of the stochastic convolution and the Kolmogorov criterion, it is postponed to Appendix~B, where the H\"older regularity of $v$ is also made precise.
\black}
\end{proof}

Recall that \( u_x = \mathcal{H}(\omega) \) in \eqref{gCLMG_eq_r}. We decompose \( \omega = v + w \). Then we have  
\[
u_x = \mathcal{H}(v) + \mathcal{H}(w), \quad  \quad   u = -\Aop^{-1/2} (v + w).\black 
\]  
We consider the following nonlinear part of equation \eqref{gCLMG_eq_r} for $a\in \mathbb{R}$.

\begin{equation}\label{NL_term_eq}
		w_{t} - \nu w_{xx}- (\mathcal{H}(v)+\mathcal{H}(w))(v+w) - a\{\Aop^{-\frac{1}{2}}(v+w)\}(v+w)_x  = 0,  \quad w(0) = \omega_0.
\end{equation}
Since $v$ is given in Lemma \ref{S_heat_wellposedness}, it is the equation of $w$. %For clarification of the notation, we write in \eqref{NL_term_eq} $u = -(-\partial^2_x)^{-1/2}(v + w)$  and then in $u_x = \mathcal{H}(v) + \mathcal{H}(w)$.%
%\begin{definition}
%    Let $m\ge 0$ and $\omega_0\in \dot{H}^m$. We say $w(t)$ is a classical solution to \eqref{NL_term_eq} if the equation \eqref{NL_term_eq} holds in $\dot{H}^{m-2}$.
%\end{definition}
We show the local existence of a solution to \eqref{NL_term_eq}.
%\textcolor{red}{Remark : There are two options for the way to take the initial data: another choice is $v(0)=\omega_0$ and $w(0)=0$. Why do we choose 
%$v(0)=0$ and $w(0)=\omega_0$ ?}

\begin{theorem}[Local existence]\label{NL_local}
Let $a\in\mathbb{R}, \nu >0$, $m\in\mathbb{N}$. Let $m_\ast \in \N$ be such that $m_\ast> m$. Assume $B_{m_{\ast}}<\infty$. \black Let $v \in X_1^m$ be the solution to the equation \eqref{S_heat_eq} given in Lemma \ref{S_heat_wellposedness}. Let $\omega_0 \in \dot{H}^m$ be $\mathcal{F}_0$-measurable. 
%that is independent of $\xi$. 
Then, there exist a random time $T^* \in (0,1]$ and a unique $(\mathcal{{F}}_t)_{t\ge 0}$-adapted solution $w \in X_{T^*}^m$ to the equation \eqref{NL_term_eq} a.s. Moreover, there exists a maximal existence time $T_{max}>0$ and $T_{max} =\infty$ or 
$\lim_{t \uparrow T_{max}} \norm{w(t)}_{\dot{H}^m} =+\infty$.
%in the sense that 
%\begin{equation}
%	w(t) = e^{-\nu t\Aop}\omega_0+\int_0^te^{-\nu(t-s)\Aop}\{(u_x(v+w))(s)+2(u(v_x+w_x))(s)\}ds,
%    \quad 
%    w(0)=\omega_0
%\end{equation}
%holds almost surely in $\dot{H}^m$ for $t\in[0,T^*]$.
% Moreover, $w$ belongs to the solution space \[
% % B_{T^*}^m\coloneqq \{w\in X_{T^*}^m; \norm{w}_{X_{T^*}^m}\leq M\coloneqq 2\norm{\omega_0}_{X_{T^*}^m}\}\]
\end{theorem}

\begin{proof}
Let $0\le T \le 1$ and $\omega_0\in \dot{H}^m$. We define the operator $\Psi(w)$ to $\dot{H}^m$ for $w\in \dot{H}^m $ and $t \le T$ as follows.
\begin{equation} \label{eq_contract}
\begin{split}
    \Psi(w)(t) &\coloneqq e^{-\nu t\Aop}\omega_0
    + \int_0^t e^{-\nu(t-s)\Aop} \{ (u_x(v+w))(s) - a(u(v_x+w_x))(s) \} \, ds \\
    &= e^{-\nu t\Aop}\omega_0
    + \int_0^t e^{-\nu(t-s)\Aop} \{ -a(u(v+w))_x(s) + (1+a) u_x(v+w)(s) \} \, ds.
\end{split}
\end{equation}
We set $B_T^M\coloneqq \{w\in X_{T}^m; \norm{w}_{X_{T}^m}\leq M\}$ where $M>0$ will be determined later. 
We show that $\Psi$ is a contraction mapping from $B_T^M$ to $B_T^M$ for sufficiently small $T>0$.
   %Recall the isometry of the Hilbert transform from Theorem~\ref{Hilbert}.% 
   For any $w \in \dot{H}^m$, we have
    \begin{equation*}
    \begin{aligned}
        \norm{e^{-\nu t\Aop}\omega_0}_{\dot{H}^m}
        &= \norm{\mathscr{F}^{-1}e^{-\nu t k^2}|k|^m \mathscr{F}\omega_0}_{l^2} \\
        &= \norm{|k|^m e^{-\nu t k^2}\mathscr{F}\omega_0}_{l^2}
        \leq \norm{|k|^m\mathscr{F}\omega_0}_{l^2}
        = \norm{\omega_0}_{\dot{H}^m}, \\
        \norm{e^{-\nu(t-s)\Aop}w_x}_{\dot{H}^m}
        &= \norm{|k|^m|k|e^{-\nu(t-s)k^2}\mathscr{F}w}_{l^2} \\
        &\leq \sup_{k\in\mathbb{Z}} |k|e^{-\nu(t-s)k^2}
        \times \norm{|k|^m\mathscr{F}w}_{l^2}
        \lesssim \nu^{-\frac{1}{2}}(t-s)^{-\frac{1}{2}}\norm{w}_{\dot{H}^m}, \\
        \norm{u(t)}_{\dot{H}^m}
        &= \norm{u_x(t)}_{\dot{H}^{m-1}}
        = \norm{\mathcal{H}({{w}+v})}_{\dot{H}^{m-1}} \\
        &= \norm{{w(t)+v(t)}}_{\dot{H}^{m-1}}
        \leq \norm{{w(t)+v(t)}}_{\dot{H}^{m}}
        \leq \norm{w}_{X_T^m} + \norm{v}_{X_1^m}.
    \end{aligned}
\end{equation*}

    From these estimates and \eqref{Hm_product}, for $w\in B_T^M$ and for $0\le s<t \le T$, we get
\begin{align*}
    \norm{e^{-\nu(t-s)\Aop}(u(v+w))_x(s)}_{\dot{H}^m}
    &\lesssim \nu^{-\frac{1}{2}}(t-s)^{-\frac{1}{2}}    \norm{u(v+w)(s)}_{\dot{H}^m} \notag \\
    &\leq C(m)\nu^{-\frac{1}{2}}(t-s)^{-\frac{1}{2}}
    \norm{u(s)}_{\dot{H}^m}\norm{(v+w)(s)}_{\dot{H}^m} \notag \\
    &\leq C(m)\nu^{-\frac{1}{2}}(t-s)^{-\frac{1}{2}}
    \norm{(v+w)(s)}_{\dot{H}^m}^2 \notag \\
    &\lesssim C(m)\nu^{-\frac{1}{2}}(t-s)^{-\frac{1}{2}}
    (\norm{v(s)}_{\dot{H}^m}^2 + \norm{w(s)}_{\dot{H}^m}^2) \notag \\
    &\leq C(m)\nu^{-\frac{1}{2}}(t-s)^{-\frac{1}{2}}
    (\norm{v}_{X_1^m}^2 + \norm{w}_{\dot{H}^m}^2)
\end{align*}
and
\begin{align*}
        \norm{e^{-\nu(t-s)\Aop}(u_x(v+w))(s)}_{\dot{H}^m}
        &\leq \norm{(u_x(v+w))(s)}_{\dot{H}^m}\leq C(m)\norm{u_x}_{\dot{H}^m}\norm{v+w}_{\dot{H}^m}\\
        &\leq C(m)\norm{v+w}_{\dot{H}^m}^2\lesssim C(m)(\norm{v}_{X_1^m}^2+\norm{w}_{X_T^m}^2).
\end{align*}
Hence, we have
\begin{align*}
    \norm{\Psi(w)(t)}_{\dot{H}^m}
    &\leq \norm{e^{-\nu t\Aop}\omega_0}_{\dot{H}^m}
    + |a|\int_0^t \norm{e^{-\nu(t-s)\Aop}(u(v+w))_x(s)}_{\dot{H}^m} \, ds \notag \\
    &\quad + |1+a|\int_0^t \norm{e^{-\nu(t-s)\Aop}(u_x(v+w))(s)}_{\dot{H}^m} \, ds \notag \\
    &\leq \norm{\omega_0}_{\dot{H}^m}
    + C(a,m,\nu)(\norm{v}_{X_1^m}^2 + \norm{w}_{X_T^m}^2)T^{\frac{1}{2}} \notag \\
    &\quad + C(a,m)(\norm{v}_{X_1^m}^2 + \norm{w}_{X_T^m}^2)T \notag \\
    &\leq \norm{\omega_0}_{\dot{H}^m}
    + C(a,m,\nu)(1 + \norm{v}_{X_1^m}^2 + M^2)T^{\frac{1}{2}}.
\end{align*}
Taking the supremum in $t \in [0,T]$ on the left hand side, we have 
$$
        \norm{\Psi(w)}_{X_T^m}\leq 1+\|\omega_0\|_{\dot{H}^m}
        +C(a, m,\nu)(1+\norm{v}_{X_1^m}^2+M^2)
        T^{\frac{1}{2}}.
$$
    Hence, we can choose $M^\ast$ and $T^\ast$ as 
    \begin{eqnarray*}
M^\ast&=& 2(1+\|\omega_0\|_{\dot{H}^m}), \\
(T^\ast)^{\frac12} &=& \min\Big \{1,\frac{M^\ast}{4C(m,\nu)(1+\norm{v}_{X_1^m}^2+(M^\ast)^2)} \Big\} 
    \end{eqnarray*}
    so that $
    \norm{\Psi(w)(t)}_{X_T^m}\leq M^\ast$, 
    which means that $\Psi$ defines a mapping from $B_{T^\ast}^
    {M^\ast}$ to $B_{T^\ast}^{M^\ast}$. 
    
    Next, we show that $\Psi$ is a contraction mapping with a fixed $T^\ast$. 
    For $w_i\in B_{T^\ast}^{M^\ast}$ ($i=1,2$), we define 
     
    \[
        u^{(i)} \coloneqq -\Aop^{-\frac{1}{2}}(v+w_i).
    \]
    \black
    For $t \le T^\ast$, we have 
    \begin{align*}
  &\norm{u^{(1)}(v+w_1)(t) - u^{(2)}(v+w_2)(t)}_{\dot{H}^m} \notag \\
  &\quad \leq \norm{u^{(1)}(v+w_1)(t) - u^{(1)}(v+w_2)(t)}_{\dot{H}^m}
    + \norm{u^{(1)}(v+w_2)(t) - u^{(2)}(v+w_2)(t)}_{\dot{H}^m} \notag \\
  &\quad \leq C(m) \norm{u^{(1)}(t)}_{\dot{H}^m} \norm{w_1(t) - w_2(t)}_{\dot{H}^m}
    + C(m) \norm{v(t) + w_2(t)}_{\dot{H}^m} \norm{u^{(1)}(t) - u^{(2)}(t)}_{\dot{H}^m} \notag \\
  &\quad \leq   C(m) \left( \norm{v(t) + w_1(t)}_{\dot{H}^m}
    + \norm{v(t) + w_2(t)}_{\dot{H}^m} \right) \norm{w_1(t) - w_2(t)}_{\dot{H}^m} \notag \\ \black
  &\quad \lesssim C(m) \left( 1 + M^* + \norm{v}_{X_1^m} \right) \norm{w_1 - w_2}_{X_{T^*}^m},
\end{align*}
and similarly
\begin{align*}
    \norm{(u^{(1)}_{x}(v+w_1)-u^{(2)}_{x}(v+w_2))(t)}_{\dot{H}^m}\lesssim C(m)(1+M^*+\norm{v}_{X_1^m})\norm{w_1-w_2}_{X_{T^*}^m}.
\end{align*}
Hence, we obtain
\begin{align*}
    &\norm{\Psi(w_1)(t) - \Psi(w_2)(t)}_{\dot{H}^m} \notag \\
    &\leq |a|\int_0^t \norm{e^{-\nu(t-s)\Aop}(u^{(1)}(v+w_1) - u^{(2)}(v+w_2))_x(s)}_{\dot{H}^m} \, ds \notag \\
    &\quad + |1+a|\int_0^t \norm{e^{-\nu(t-s)\Aop}(u^{(1)}_{x}(v+w_1) - u^{(2)}_{x}(v+w_2))(s)}_{\dot{H}^m} \, ds \notag \\
    &\leq C(a, m, \nu)
    \norm{u^{(1)}(v+w_1) - u^{(2)}(v+w_2)}_{X_{T^*}^m}
    \int_0^t (t-s)^{-\frac{1}{2}} \, ds \notag \\
    &\quad + C(m,a)\int_0^t
    \norm{(u^{(1)}_{x}(v+w_1) - u^{(2)}_{x}(v+w_2))(s)}_{\dot{H}^m} \, ds \notag \\
    &\leq C(a,m,\nu){T^*}^{\frac{1}{2}}
    (1 + M^* + \norm{v}_{X_1^m})\norm{w_1 - w_2}_{X_T^m} \notag \\
    &\quad + C(a,m)T^*
    \norm{u^{(1)}_{x}(v+w_1) - u^{(2)}_{x}(v+w_2)}_{X_{T^\ast}^m} \notag \\
    &\leq C(a, m, \nu){T^\ast}^{\frac{1}{2}}
    (1 + M^\ast + \norm{v}_{X_1^m})\norm{w_1 - w_2}_{X_T^m}.
\end{align*}
taking $T^{\ast}$ sufficiently small (we denote by the same letter) such that 
$ C(m,\nu)({T^\ast}^{\frac{1}{2}})(1+M^\ast+\norm{v}_{X_1^m}) \le \frac12,$ we have
$$
    \norm{\Psi(w_1)(t)-\Psi(w_2)}_{X_{T^\ast}^m} \le \frac12 \norm{w_1-w_2}_{X_{T^\ast}^m}.
$$
This implies that $\Psi:B_{T^\ast}^{M^\ast}\to B_{T^\ast}^{M^\ast}$ is a contraction mapping. 
Consequently, according to the Banach fixed point theorem, there exists a unique $w\in B_{T^\ast}^{M^\ast}$ that satisfies $w=\Psi(w)$. 
We have obtained the unique mild solution in $B_{T^\ast}^{M^\ast}$ which is a subset of $X_{T^\ast}^m$. 
In fact, this is the only mild solution in the whole space $X_{T^\ast}^m$. Suppose $w_1, w_2\in X_{T^\ast}^m$ are the mild solutions for the same $\omega_0$ and $v$. 
Similarly as above, we obtain, for $0 < T\le T^\ast$, 
    \begin{align*}
       \norm{w_1-w_2}_{X_{T}^m} 
       &=\norm{\Psi(w_1)-\Psi(w_2)}_{X_T^m}\\
       &\leq C(m, \nu)T^{\frac12} (1+R+\norm{v}_{X_1^m})
       \norm{w_1-w_2}_{X_{T}^m}
    \end{align*}
    with here $R:= \max\{\norm{w_1}_{X_{T^\ast}^m}, \norm{w_2}_{X_{T^\ast}^m} \}$. We then deduce that $w_1=w_2$ on $[0,T^{\ast\ast}]$ by choosing 
    \[
    C(m, \nu)(T^{\ast\ast})^{\frac12} (1+R \norm{v}_{X_1^m})=\frac12 .
    \] Iterating this procedure $[T^{\ast\ast}, 2T^{\ast\ast}], [2T^{\ast\ast}, 3T^{\ast\ast}],\cdots,$
    we have $w_1=w_2$ in $X_{T^\ast}^m$. 
    Finally, we define 
    \[
    T_{max} :=\sup\{t>0, \mbox{there exists a unique solution to {\eqref{eq_contract}} such that }  \norm{w}_{X_{t}^m} <\infty\}.
    \]
    If $T_{max}=+\infty$, the solution exists globally. Thus, we consider 
    the case $T_{max}<+\infty$ and assume  $\lim_{t \uparrow T_{max}} \norm{w(t)}_{\dot{H}^m} <\infty$. In this case, we can find a small $\eta>0$ such that  
    $\norm{w(T_{max}-\tau)}_{\dot{H}^m} <\infty$. 
    By the above argument, we can solve the equation starting from $w(T_{max}-\tau)$ on $[0,T^\ast]$. This constructs a solution on $[0, T_{max} -\tau +T^\ast]$, which contradicts the definition of 
    $T_{max}$.
    \end{proof}
  
The following lemma is the same result as proved in~\cite{tsuji2023statistical}.
\begin{lemma}[Continuous dependence of  mild \black solution with respect to the initial data]\label{thm:conti_dependence}.
Let $a\in \R$, $\nu >0$ and $m\in\mathbb{N}$.  
Let $m_\ast \in \N$ such that $m_\ast> m$ and assume $B_{m_{\ast}}<+\infty$. \black Let $0<T<T_{max}$, where $T_{max}$ is the maximal existence time defined in Theorem \ref{NL_local}. \black   
%Suppose that $f_{1},f_{2} \in X_{T}^{m}$ and $\omega_0^{(1)},\omega_0^{(2)} \in \dot{H}^{m}$. 
Suppose that $w_{i} \in X_{T}^{m}$, $i=1,2$ represents the  mild \black  solution of \eqref{NL_term_eq} for the initial data $\omega_{0i}\in \dot{H}^m$  being $\mathcal{F}_0$-measurable.
Then, there exists a constant  $C(a, \nu,T,\norm{w_{1}}_{X_{T}^{m}},\norm{w_{2}}_{X_{T}^{m}}, \norm{v}_{X_T^m})\ge0$ \black such that the following inequality holds.
\begin{equation}
	\norm{w_{1}-w_{2}}_{X_{T}^{m}}\leq C \snorm{\omega_0^{(1)}-\omega_0^{(2)}}_{\dot{H}^{m}}.
	 \label{conti_dependence}
\end{equation}
\end{lemma}
\begin{proof}
    In a similar way as in the proof of Theorem \ref{NL_local}, 
    since\begin{align*}
        \norm{w_1(t)-w_2(t)}_{\dot{H}^{m}}
        &=\norm{\Psi(w_1)(t)-\Psi(w_2)(t)}_{\dot{H}^{m}}\\
        &\leq \snorm{\omega_0^{(1)}-\omega_0^{(2)}}_{\dot{H}^{m}} \\
        &  +C(a, m, \nu, \norm{w_1}_{X_T^m}, \norm{w_2}_{X_T^m}, \norm{v}_{X_T^m})\\
        &\qquad\times\int_0^t (1 + (t-s)^{-\frac{1}{2}})\norm{w_1(s)-w_2(s)}_{\dot{H}^m}ds.
        \end{align*}
        Gronwall's inequality yields
        \[
            \norm{w_1(t)-w_2(t)}_{\dot{H}^{m}}\leq C\snorm{\omega_0^{(1)}-\omega_0^{(2)}}_{\dot{H}^{m}}.
        \]
 Taking the supremum on $0\leq t\leq T$, we finish the proof.
\end{proof}

To prove an a priori estimate for the global solution, we introduce the projection operator $P_{N}\colon L^2(\mathbb{S}^1) \to \cap_{m=1}^{\infty} \dot{H}^{m}(\mathbb{S}^{1})$ by 
\[
P_{N}f \coloneqq \sum_{0<\abs{k}\leq N} f_k e_k.
\] 
Note that $P_{N}$ is a bounded linear operator for each $N \in \N$.
We consider the Galerkin approximation $w^N$ as the solution of the following equation.
\begin{equation}\label{galerkin_ode}
    \begin{aligned}
		&\partial_{t}w - \nu w_{xx} - P_{N}((P_N u)_{x} (P_N v+P_N w) - a P_N u (P_N v+ P_N w)_{x})=0, \\& \quad  (P_N u)_{x} = \mathcal{H}(P_N v +P_N w),\quad w^{N}(0)=P_N \omega_0,
\end{aligned}
\end{equation}
where $v$ is the solution of \eqref{S_heat_eq}. 
%Since it gives rise to the $2N$-dimensional ordinary differential equation for the coefficients $\widehat{u}(n)$, there exists a smooth classical solution, say $w^N$, locally in time. 
Let $T_{max}^N$ denote its maximal existence time. 
The following lemma gives the convergence of $P_N v$ to $v$ in $X_T^m$ a.s. for any $T>0$.  
\begin{lemma}\label{v^N_conv}  Let $\nu>0$ and $N,m\in \N$.
Let $m_\ast \in \N$ be such that $m_\ast> m$. Assume $B_{m_{\ast}}<+\infty$.  \black Let $T>0$ and  $v \in X^m_T$ be the solution of (\ref{S_heat_eq}). \black Set $v_N=P_N v$. Then, there exists a small $\delta \in (0,1)$ such that  
$$\E\left[\|v_N-v\|^2_{X_T^m}\right] \lesssim N^{-\delta}. $$
\end{lemma}
\begin{proof}
Let {$f_N(t)=v_{N}(t)-v(t)$.} 
It suffices to check that for small $\delta$ and $\gamma$, we have
$$
\E\left[\|f_N(t)\|^2_{\dot{H}^{m}}\right] \lesssim N^{-\delta}, \qquad \E\left[\|f_N(t)-f_N(s)\|^2_{\dot{H}^m}\right]
\lesssim N^{-\delta} (t-s)^{\gamma}.
$$
These estimates are verified as in Lemma 2.1. 
%\textcolor{blue}{(Before : Thus, $(v_N)$ is Cauchy in $L^2(\Omega; X_T^m)$, and there exists a limit $v \in L^2(\Omega; X_T^m)$ that should satisfy the linear equation.)%
We have the desired
estimate by the Kolmogorov test or the Garsia-Rodemich-Rumsey inequality.
\end{proof}
We can conclude the convergence almost surely by the above lemma. Indeed, 
it follows from the standard martingale inequality and hyper-contractivity (See Theorem 1.1 of \cite{nualart}) that 
\begin{eqnarray*}
\mathbb{P}(\|v_N-v\|_{X_T^m} >R) 
&\le& R^{-p} \sup_{t\in [0,T]} \E [\|v_N-v\|^p_{\dot{H}^m}] \\
&\lesssim &p^{\frac{p}{2}} R^{-p} 
(\E[\|v_N -v\|^2_{\dot{H}^m}])^{\frac{p}{2}}
\lesssim p^{\frac{p}{2}} R^{-p} 
N^{-\frac{\delta}{2} p}.
\end{eqnarray*}
Optimizing in $p$ implies
$$ \mathbb{P}(\|v_N-v\|_{X_T^m} >R)
\le C e^{-cR^2 N^{\delta}}.$$
For $M\ge 1$, we define 
$\Sigma_M \subset \Omega$ as
$$\Sigma_M =\left\{\eta\in \Omega \; \left\vert\; v \in X_T^m,  \forall N\ge 1, \|v_N-v\|_{X_T^m}\le M N^{-\frac{\delta}{2}}\right.\right\}.$$
Then, we have $$\mathbb{P}(\Omega \setminus \Sigma_M) \le C e^{-cM^2}.$$
Hence, $v_N \to v$ in $X^m_T$ on $\Sigma_M$. Finally, the Borel--Cantelli lemma shows $\mathbb{P}(\Sigma)=1,$ where
$\Sigma =\liminf_{M\to \infty} \Sigma_M$. 

Next, we establish the uniform energy estimates for $w^N$.

\begin{lemma}[A priori estimate for Galerkin approximation]\label{apriori}
Let $a = -2, 0<T<\infty$, $\nu>0, m\in\mathbb{N}$. Let $m_\ast \in \N$ be such that $m_\ast> m$. Assume that $B_{m_{\ast}}<+\infty$ and $\omega_0\in\dot{H}^m$ is $\mathcal{F}_0$-measurable. 
Then, the solution $w^N$ of \eqref{galerkin_ode} satisfies the following estimates for $t\in [0,T^N_{max} \wedge T]\colon$

\begin{equation}\label{L2_estimate_wN}
    \norm{w^N(t)}_{H}\leq C(\norm{v}_{X_T^1}, T, \norm{\omega_0}_{H}), 
\end{equation}
\begin{equation}\label{Hm_estimate_for_w}
    \norm{w^N (t)}_{\dot{H}^m}\leq {\revcol C(\nu, \norm{v}_{X_T^m}, T, \norm{\omega_0}_{\dot{H}^m})}, 
\end{equation}
\begin{equation}\label{apriori_m+1_estimate}
    \norm{w^N(t)}_{\dot{H}^m}^2+\nu\int_0^t \norm{w^N(t)}_{\dot{H}^{m+1}}^2 ds \leq {\revcol C(\nu, \norm{v}_{X_T^m}, T, \norm{\omega_0}_{\dot{H}^m})}.
\end{equation}
\end{lemma}
\begin{proof} 
In \eqref{galerkin_ode}, multiplying $w^N \in C^1([0,T];C^{\infty}(\mathbb{S}^1))$ \black and taking the integration with respect to $x$, 
since $P_N w^N =w^N$, we have
\begin{align}\label{L2_estimate}
	\frac{1}{2}\frac{d}{dt}\norm{w^N(t)}_{H}^{2} + \nu \norm{w^N(t)}_{\dot{H}^{1}}^{2}	= \iprod{(P_Nu)_{x}(P_Nv+w^N) + 2(P_Nu)(P_Nv+w^N)_{x},w^N}_{H}.
\end{align}
The function $u$ is divided into two parts, such as $u=u_v+u_w$ with $(P_Nu_{v})_x = \mathcal{H}(P_Nv)$ and $(P_Nu_{w}) _x=\mathcal{H}(w^N)$. Then, the inner product on the right-hand side becomes
\begin{align*}
  &\iprod{(P_N u)_{x}(P_N v + w^N) 
  + 2(P_N u)(P_N v + w^N)_{x},\, w^N}_{H} \\
  &= \iprod{(P_N u_v)_x (P_N v),\, w^N}_{H} 
  + \iprod{(P_N u_w)_x (P_N v),\, w^N}_{H} \\
  &\quad + \iprod{(P_N u_v)_x w^N,\, w^N}_{H} 
  + \iprod{(P_N u_w)_x w^N,\, w^N}_{H} \\
  &\quad + 2\iprod{(P_N u_v)(P_N v)_{x},\, w^N}_{H} 
  + 2\iprod{(P_N u_w)(P_N v)_{x},\, w^N}_{H} \\
  &\quad + 2\iprod{(P_N u_v) w^N_{x},\, w^N}_{H} 
  + 2\iprod{(P_N u_w) w^N_{x},\, w^N}_{H} \\
  &= \iprod{(P_N u_v)_x (P_N v),\, w^N}_{H} 
  + \iprod{(P_N u_w)_x (P_N v),\, w^N}_{H} \\
  &\quad + 2\iprod{(P_N u_v)(P_N v)_{x},\, w^N}_{H} 
  + 2\iprod{(P_N u_w)(P_N v)_{x},\, w^N}_{H}
\end{align*}
since
\begin{align*}
&\iprod{(P_N u_v)_x w^N, w^N}_{H}
+ 2\iprod{(P_N u_v) w^N_x, w^N}_{H} = 0, \\
&\iprod{(P_N u_w)_x w^N, w^N}_{H}
+ 2\iprod{(P_N u_w) w^N_x, w^N}_{H} = 0
\end{align*}
owing to integration by parts. In addition, {  by Proposition \ref{Hilbert} and}
the Sobolev embedding $\dot{H}^{1} \subset L^{\infty}$,
\begin{align*}
	&\abs{\iprod{(P_N u_v)_x(P_Nv),w^N}_{H}}
    \leq \norm{P_Nv}_{L^\infty} \norm{\mathcal{H}(P_Nv)}_{H} \norm{w^N}_{H} \leq \norm{P_Nv}_{\dot{H}^{1}} \norm{P_Nv}_{H} \norm{w^N}_{H}, \\
	&\abs{\iprod{(P_N u_v)(P_Nv)_{x},w^N}_{H}}
\leq \norm{P_Nu_v}_{L^\infty}\norm{P_N v}_{\dot{H}^1} \norm{w^N}_{H} \leq\norm{P_Nv}_{\dot{H}^{1}} \norm{P_Nv}_{H} \norm{w^N}_{H},
\end{align*}
where $\norm{P_Nu_v}_{\dot{H}^1} = \norm{(P_Nu_v)_x}_{H} = \norm{\mathcal{H}(P_Nv)}_{H} = \norm{P_Nv}_{H}$. Similarly, we have
\begin{align*}
	\abs{\iprod{(P_N u_w)_x(P_N v),w^N}_{H}} &\leq \norm{P_Nv}_{\dot{H}^1} \norm{w^N}_{H}^{2}, \\
\abs{\iprod{(P_Nu_{w})(P_Nv)_{x},w^N}_{H}} &\leq\norm{P_Nv}_{\dot{H}^1} \norm{w^N}_{H}^{2}.
\end{align*}
Noting $\norm{P_N v}_{H}\leq \norm{P_N v}_{\dot{H}^1}\leq \norm{v}_{\dot{H}^1}$, then \eqref{L2_estimate} becomes

\begin{equation*}
	\frac{1}{2}\frac{d}{dt}\norm{w^N(t)}_{H}^{2} \leq 	\frac{1}{2}\frac{d}{dt}\norm{w^N(t)}_{H}^{2} + \nu \norm{w^N(t)}_{\dot{H}^{1}}^{2} \leq C(\norm{v}_{X_{T}^{1}})  \norm{w^N(t)}_{H}+ C(\norm{v}_{X_{T}^{1}}) \norm{w^N(t)}_{H}^{2},
\end{equation*}

which yields $\norm{w^N(t)}_{H} \leq C(\norm{v}_{X_T^1},T, \norm{\omega_0}_{H})$ for $t \in [0,T^N_{max} \wedge T]$ by Gronwall's inequality.
We write $v_N$ instead of $P_N v$ as in the previous lemma.
Taking $\partial_x^m$ from equation \eqref{galerkin_ode} and taking the inner product $L^2$ with $\partial_{x}^{m}w^N$, we have
\begin{equation*}
	 \begin{split}
 &\frac{1}{2}\frac{d}{dt} \norm{w^N(t)}_{\dot{H}^{m}}^{2} + \nu \norm{w^N(t)}_{\dot{H}^{m+1}}^{2}\\
 &\qquad= -2 \iprod{(P_Nu)(  v_N + w^N),w^N_x}_{\dot{H}^{m}} - \iprod{(P_Nu)_{x}(v_N + w^N), w^N}_{\dot{H}^{m}}.
 \end{split}
\end{equation*}
We estimate the inner products on the right side using \eqref{Hm_product} and $\norm{u_v}_{\dot{H}^m}=\norm{\partial_x u_v}_{\dot{H}^{m-1}} = \norm{\mathcal{H}(v)}_{\dot{H}^{m-1}} = \norm{v}_{\dot{H}^{m-1}}$
as follows.
\begin{align*}
\abs{\iprod{(P_N u)(v_N + w^N), w^N_x}_{\dot{H}^{m}}} 
&\leq \norm{(P_N u)(v_N + w^N)}_{\dot{H}^{m}} 
      \norm{w^N}_{\dot{H}^{m+1}} \\
&= \norm{(P_N u_v + P_N u_w)(v_N + w^N)}_{\dot{H}^{m}} 
    \norm{w^N}_{\dot{H}^{m+1}} \\
&\leq \norm{P_N u_v + P_N u_w}_{\dot{H}^{m}} 
      \norm{v_N + w^N}_{\dot{H}^{m}} 
      \norm{w^N}_{\dot{H}^{m+1}} \\
&\leq (\norm{v}_{\dot{H}^{m-1}} + \norm{w^N}_{\dot{H}^{m-1}})
     (\norm{v}_{\dot{H}^{m}} + \norm{w^N}_{\dot{H}^{m}})
     \norm{w^N}_{\dot{H}^{m+1}} \\
&\leq C(v) \norm{w^N}_{\dot{H}^{m+1}} 
      + C(v) \norm{w^N}_{\dot{H}^{m}} \norm{w^N}_{\dot{H}^{m+1}} \\
&\quad + C(v) \norm{w^N}_{\dot{H}^{m-1}} \norm{w^N}_{\dot{H}^{m+1}} \\
&\quad + \norm{w^N}_{\dot{H}^{m-1}} \norm{w^N}_{\dot{H}^{m}} \norm{w^N}_{\dot{H}^{m+1}}, \\     
\\
\intertext{and}
\abs{\iprod{(P_N u)_{x}(v_N + w^N), w^N}_{\dot{H}^{m}}}
&\leq \norm{(P_N u)_{x}(v_N + w^N)}_{\dot{H}^{m}} 
      \norm{w^N}_{\dot{H}^{m}} \\
&\leq \norm{\mathcal{H}(v_N + w^N)}_{\dot{H}^{m}} 
      \norm{v_N + w^N}_{\dot{H}^{m}} 
      \norm{w^N}_{\dot{H}^{m}} \\
&\leq (\norm{v}_{\dot{H}^{m}} + \norm{w^N}_{\dot{H}^{m}})^2 
      \norm{w^N}_{\dot{H}^{m}} \\
&\leq C(v) \norm{w^N}_{\dot{H}^{m}} 
      + C(v) \norm{w^N}_{\dot{H}^{m}}^2 
      + \norm{w^N}_{\dot{H}^{m}}^3.
\end{align*}

Hence, we have
\begin{align*}
    &\frac{1}{2} \frac{d}{dt} \norm{w^N(t)}_{\dot{H}^{m}}^{2} 
    + \nu \norm{w^N(t)}_{\dot{H}^{m+1}}^{2} \\
    &\leq C(v) \big(
        \norm{w^N(t)}_{\dot{H}^{m+1}}
        + \norm{w^N(t)}_{\dot{H}^{m}} \norm{w^N(t)}_{\dot{H}^{m+1}} \\
    &\quad
        + \norm{w^N(t)}_{\dot{H}^{m-1}} \norm{w^N(t)}_{\dot{H}^{m+1}} 
        + \norm{w^N(t)}_{\dot{H}^{m-1}} \norm{w^N(t)}_{\dot{H}^{m}} \norm{w^N(t)}_{\dot{H}^{m+1}} \\
    &\quad
        + \norm{w^N(t)}_{\dot{H}^{m}}
        + \norm{w^N(t)}_{\dot{H}^{m}}^{2} 
        + \norm{w^N(t)}_{\dot{H}^{m}}^{3}
    \big).
\end{align*}

Sobolev's interpolation \eqref{Sobolev_interpolation} and Young's inequality yield 
\begin{align*}
    & \norm{w^N(t)}_{\dot{H}^{m+1}} \leq \frac{\nu}{14}\norm{w^N(t)}_{\dot{H}^{m+1}}^2 + C(\nu),\\
      &  \norm{w^N(t)}_{\dot{H}^{m}} \norm{w^N(t)}_{\dot{H}^{m+1}}
    \leq  \norm{w^N}_{\dot{H}^{m-1}}^{\frac{1}{2}} \norm{w^N}_{\dot{H}^{m+1}}^{\frac{3}{2}}
    \leq \frac{\nu}{14} \norm{w^N}_{\dot{H}^{m+1}}^{2} + C(\nu) \norm{w^N}_{\dot{H}^{m-1}}^{2}, \\
     & \norm{w^N(t)}_{\dot{H}^{m-1}} \norm{w^N(t)}_{\dot{H}^{m+1}}
    \leq \frac{\nu}{14} \norm{w^N}_{\dot{H}^{m+1}}^{2} + C(\nu) \norm{w^N}_{\dot{H}^{m-1}}^{2}, \\
    & \norm{w^N(t)}_{\dot{H}^{m-1}} \norm{w^N(t)}_{\dot{H}^{m}} \norm{w^N(t)}_{\dot{H}^{m+1}}
    \leq \frac{\nu}{14} \norm{w^N}_{\dot{H}^{m+1}}^{2} + C(\nu) \norm{w^N}_{\dot{H}^{m-1}}^{6}, \\
    &\norm{w^N}_{\dot{H}^{m}} \leq  \norm{w^N}_{\dot{H}^{m-1}}^{\frac{1}{2}} \norm{w^N}_{\dot{H}^{m+1}}^{\frac{1}{2}}
    \leq \frac{\nu}{14} \norm{w^N}_{\dot{H}^{m+1}}^{2} + C(\nu) \norm{w^N}_{\dot{H}^{m-1}}^{\frac{2}{3}}, \\
    & \norm{w^N(t)}_{\dot{H}^{m}}^{2} 
    \leq  \norm{w^N}_{\dot{H}^{m-1}} \norm{w^N}_{\dot{H}^{m+1}}
    \leq \frac{\nu}{14} \norm{w^N}_{\dot{H}^{m+1}}^{2} + C(\nu) \norm{w^N}_{\dot{H}^{m-1}}^{2}, \\ 
    &  \norm{w^N(t)}_{\dot{H}^{m}}^{3} 
    \leq  \norm{w^N}_{\dot{H}^{m-1}}^{\frac{3}{2}} \norm{w^N}_{\dot{H}^{m+1}}^{\frac{3}{2}}
    \leq \frac{\nu}{14} \norm{w^N}_{\dot{H}^{m+1}}^{2} + C( \nu) \norm{w^N}_{\dot{H}^{m-1}}^{6}.
\end{align*}
We finally obtain 
\begin{equation}\label{dif_estimate_wN}
\begin{aligned}
  &\frac{1}{2} \frac{d}{dt} \norm{w^N(t)}_{\dot{H}^{m}}^{2}
  + \frac{\nu}{2} \norm{w^N(t)}_{\dot{H}^{m+1}}^{2}\\
  &\leq\  C(\nu, v) \big(
      \norm{w^N(t)}_{\dot{H}^{m-1}}^{2/3}
    + \norm{w^N(t)}_{\dot{H}^{m-1}}^{2}
    + \norm{w^N(t)}_{\dot{H}^{m-1}}^{6}
  \big) 
   + C\big(\nu, \norm{v}_{X_T^m}\big).
\end{aligned}
\end{equation}
\black
For $m=1$, it follows from $\norm{w^N(t)}_{L^{2}} \leq C(\norm{v}_{X_{T}^{1}}, T, \norm{\omega_0}_{H})$ and Gronwall's inequality that we have the estimate $\norm{w^N(t)}_{\dot{H}^{1}} \leq C(v, \nu, T, \norm{\omega_0}_{\dot{H}^1})$.
Repeating this process inductively for $m\geqq 2$, we have an a priori estimate $\norm{w^N(t)}_{\dot{H}^{m}} \leq C(v,\nu, T, \norm{\omega_0}_{\dot{H}^m})$ in $\dot{H}^m$. 
The estimate \eqref{apriori_m+1_estimate} is obtained by integrating \eqref{dif_estimate_wN} and using \eqref{Hm_estimate_for_w}.
\end{proof}

Now we see that the local-in-time solution obtained in Theorem \ref{NL_local} satisfies the a priori estimate \eqref{Hm_estimate_for_w}, allowing $N$ to go to infinity.
Thus, the local-in-time solution can be extended to the global solution. 
Using the same arguments as in Theorem 2.2, we conclude that a short time $T^\ast_N >0$ and the existence of a unique solution $w^N$ are ensured. By Lemma \ref{v^N_conv} and the continuity of $T_N^\ast$ with respect to $v_N$, we have $T_N^\ast \to T^\ast$ a.s. 
Moreover, for $t< T_N^\ast$, $\|w^N\|_{X^m_t} \le M^\ast_N$. 
Since $M^\ast$ is continuous with respect to $v$, we have $\|w^N\|_{X^m_t} \le 2M^\ast$ for large $N$ and $t< T_N^\ast$ similarly to the proof of Theorem 2.2. 
Thus, for $t < T_{max} \wedge T_N^\ast,$ we have
\begin{eqnarray*}
\|w^N -w\|_{X_t^m} 
\lesssim t^\theta C(\nu, m,  \|v\|_{X_1^m}, \|\omega_0\|_{\dot{H}^m})
\{\|\omega_0^N-\omega_0\|_{\dot{H}^m}+\|v_N-v\|_{X^m_1} +\|w^N-w\|_{X_t^m}\}
\end{eqnarray*}
for some $\theta>0$. 
Consequently, by the convergence of $v^N$,  we can choose a small $t^\ast<t$ such that $\| w^N - w\|_{X_s^m} \to 0$ for each $s\le t^\ast$. 
Repetition of this procedure on $[0,t^\ast],[t^\ast,2t^\ast],\cdots$, we have the convergence on $[0,t]$.

Let $w(t)\in \dot{H}^m$ be the local mild solution for $t<T_{max}$ for the initial data $\omega_0\in \dot{H}^m$ obtained in Theorem \ref{NL_local}. 
By convergence $\|w^N - w\|_{X_t^m} \to 0$, $w^N$ exists on $[0,t]$, and thus we have $t < T_{max}^N$ for large $N$. 
Therefore, $T_{max} \le \liminf_{N\to \infty} T_{max}^N,$ and taking the limit in \eqref{Hm_estimate_for_w}, we have
\[
    \norm{w(t)}_{\dot{H}^m}\leq {\revcol C(\nu,\norm{v}_{X_{T}^m},\norm{\omega_0}_{\dot{H}^m}, T)}  
\]
for $t\in[0, T_{max} \wedge T]$.

From the results obtained so far, we can establish the existence of a global solution to the gCLMG equation \eqref{gCLMG_eq_r}. Moreover, we consider a projected equation
\begin{equation}
\partial_{t}\omega - \nu \omega_{xx} - P_{N}((P_N u)_{x} P_N\omega -a P_N u (P_N\omega)_{x})=P_N \xi_{t},\quad (P_N u)_{x} = \mathcal{H}(P_N \omega),\quad \omega(0)=P_{N}\omega_0.
\label{gCLMG_eq_proj}
\end{equation}

\begin{theorem}\label{sgclmg_global}
Let $\nu >0$ $m\in \N$ and $a=-2$. Let $m^\ast\in \N$ be such that  $m_\ast>m$. Assume that $B_{m_{\ast}}<+\infty$ and $\omega_{0} \in \dot{H}^{m}$ is $\mathcal{F}_0$-measurable. Then, the equation \eqref{gCLMG_eq_r}, 
\begin{align}\label{SPDE_gCLMG}
\omega_{t} +a u\omega_{x} - u_{x} \omega  - \nu \omega_{xx} = \xi_t, \quad u_x = \mathcal{H}(\omega), \quad \omega(0) = \omega_{0}
\end{align}
has a unique $(\mathcal{F}_t)_{t\ge 0}$-adapted solution $\omega\in  X_\infty^m$ a.s. 
Moreover, for each $N \in \N$, there exists a unique solution to \eqref{gCLMG_eq_proj}, 
 denoted by $\omega^N$. In addition, it converges to the solution $\omega$ of the
stochastic gCLMG equation \eqref{gCLMG_eq_r} in $X_{T}^{m}$.
\end{theorem} 

\begin{proof} 
Let $v$ denote the solution to equation \eqref{S_heat_eq} established by Lemma~\ref{S_heat_wellposedness}. 
Set $v^{N} = P_{N}v$. 
Then, $v^{N}$  is the solution to $v_t^N - \nu {v}_{xx}^N = P_N \xi_t$ and $v^N(0) = 0$. 
For this solution $v^N$, we have the solution $w^{N}$ to \eqref{galerkin_ode} and it holds that 
\begin{equation}\label{wn_T_conv}
   \lim_{N \to \infty} \norm{w-w^N}_{X_{T}^m}\to 0
\end{equation} as in the previous section.
By Lemma 2.4, $v^N$ converges to $v$ in $X_T^m$ almost surely. 
Hence, $\omega^N := v^N + w^N$ is the solution to \eqref{gCLMG_eq_proj}, which is unique by the same arguments as in Theorem~\ref{NL_local} and Lemma~\ref{thm:conti_dependence}. 
Thus, it converges to the solution $\omega=v+w$ to the stochastic gCLMG equation \eqref{gCLMG_eq_r} in $X_{T}^{m}$ a.s.
\end{proof}

\section{Existence of an invariant measure} 

We obtain a uniform $\dot{H}^m$ estimate for the solution, thereby showing the existence of an invariant measure for the stochastic gCLMG equation \eqref{gCLMG_eq_r} by the Krylov-Bogoliubov argument in a similar way to Boritchev and Kuksin \cite{Boritchev_Kuksin_2021_onedimensional}. 
We denote the solution of \eqref{gCLMG_eq_r} with a (deterministic) initial data $y \in \dot{H}^m$ by $\omega(t; y)$. Let $(P_t)_{t\ge0}$ be the transition semigroup for \eqref{gCLMG_eq_r} with $a=-2$, which is defined by
\[
    P_t f (y) = \E \left[f(\omega(t; y))\right], \quad  f \in B_b(\dot{H}^m),\quad y\in \dot{H}^m.
\]

%\begin{remark}
%    Due to \eqref{conti_dependence}, it is easy to check that the %solution $\omega$ has Markov property.
%\end{remark}

%Let $\mu_{\xi}$ and $\mu_{\omega_0}$ be the distributions of $\xi$ and %$\omega_0$ respectively. Remark that the solution $\omega (t)$ depends %on $\xi$ and $\omega_0$, and thus we write the solution map $\omega(t, %\eta) =\mathcal{M}_t (\omega_0(\eta_1), \xi ({\eta_2}))$. 
%Since $\xi$ and $\omega_0$ are %independent, 
%For nonnegative measurable functions
%$f\colon \dot{H}^m \to \R$, Fubini's Theorem yields the following result.
%\begin{eqnarray}\label{Fubini}
%\E[f(\omega(t))] &=& 
%\int \int f(\mathcal{M}_t(y, \zeta)) \mu_{\omega_0}(dy)\mu_{\xi}(d\zeta) %\\
%&=&
%\int \mu_{\omega_0}(dy) 
%\int f(\mathcal{M}_t(y, \zeta)) 
%\mu_{\xi}(d\zeta) \\
%&=&
%\int_{\dot{H}^m} \E_{\xi}[f(\omega(t;y))] \mu_{\omega_0}(dy).
%\end{eqnarray}

We consider a random initial data $\omega_0$, which is $\mathcal{F}_0$-measurable whose law is $\mu_0$.
Since we assume that $\sigma\{\beta_\tau-\beta_\sigma: \tau\ge \sigma\ge s\}$ is independent of $\mathcal{F}_s$ for any $s\ge 0$, the solution of \eqref{gCLMG_eq_r} has the Markov property. 
%\footnote{For a random variable $X$, we denote by $\sigma(X)$ the $\sigma$-algebra generated by $X$.}% 
Thus, for $\omega_0$, which is  $\mathcal{F}_0$-measurable, we have
$$P_t f(\omega_0) =\E[f(\omega(t;\omega_0))| {\sigma(\omega_0)}]\qquad \mathbb{P}\text{-a.s.}$$
and 
\begin{align*}
\E[f(\omega(t;\omega_0))] 
&= \E[\E[f(\omega(t;\omega_0))|\sigma(\omega_0)]] =\E[P_t f (\omega_0)] \\
 &= { \int_{\dot{H}^m} P_t f (x) \mu_0(dx)= \int_{\dot{H}^m} \E(f(\omega(t;x)) \mu_0(dx)}.
\end{align*}

Hence, it is sufficient to establish the estimates (like  \eqref{Hm_expectation_estimate} and \eqref{Hm_integration_estimate} below) for non-random initial data $\omega_{0} \in \dot{H}^m$ and then take the expectation with respect to the distribution $\mu_0$. 
In what follows, we deal with initial deterministic data $\omega_0 \in \dot{H}^m$. 
\vspace{3mm}

%We fix $m\in\mathbb{N}$.
\begin{theorem}\label{Stochastic_Hm_estimate} 
Let $m\in \N$, $a = -2, \nu >0$ and let $m_\ast\in \N$ with $m_\ast>m$ such that $B_{m^\ast} <\infty$. Let $\omega_0 \in \dot{H}^m$. 
%Assume also that $\E[e^{\sigma \norm{\omega_0}_{H}^2}]<\infty$ 
For $\sigma>0$ with $\sigma B_0 \le \nu$,   
%and $\E[\norm{\omega_{0}}_{\dot{H}^{m}}^{2}] < \infty$.Then 
there exist constants $C(m,\nu,\sigma)$ and $ C'(m,\nu,\sigma)>0$ such that the solution to equation \eqref{gCLMG_eq_r} satisfies 
\begin{equation}\label{Hm_expectation_estimate}
\E[\norm{\omega(t)}_{\dot{H}^m}^2] \leq C (1+\norm{\omega_0}_{\dot{H}^m}^2+ e^{\sigma^\prime \norm{\omega_0}_{H}^2})
\end{equation} and
\begin{equation}\label{Hm_integration_estimate}
\E\left[\int_0^t\norm{\omega(s)}_{\dot{H}^{m+1}}^2 ds\right] \leq \frac{1}{\nu}\norm{\omega_0}_{\dot{H}^m}^2 + t C'\left(1+ e^{\sigma^\prime\norm{\omega_0}_{H}^2}\right)
\end{equation}
\black for any $t\geqq 0$, where $\sigma^\prime \coloneqq \sigma/3$.
\end{theorem} 

\begin{proof}
%Let $\mu_{\xi}$ and $\mu_{\omega_0}$ be the distributions of $\xi$ and $\omega_0$ respectively.  Since $\xi$ and $\omega_0$ are independent,  for non-negative measurable functions
%$f\colon \dot{H}^m \to \R$, Fubini's theorem yields
%\begin{equation}\label{Fubini}
%\E[f(\omega(t;\omega^\eta_0,\xi^\eta))] %= \int_{\dot{H}^m} \E[f(\omega(t;y))] %\mu_{\omega_0}dy.
%\end{equation}
%Hence, it is sufficient to establish the estimate \eqref{Hm_expectation_estimate} and \eqref{Hm_integration_estimate} for non-random initial data $\omega_{0} \in \dot{H}^m$ and then to take the expectation.

For $\omega_0 \in \dot{H}^m$, it follows from Theorem~\ref{sgclmg_global} that there exists a solution $\omega^N$ to the equation
\[
\partial_t\omega^N - \nu \omega_{xx}^N - P_N(u_x^N \omega^N + 2u^N \omega_x^N)= P_N \xi_t,\quad u^N_x = \mathcal{H}(\omega^N), \quad	\omega^N(0)=P_N\omega_0.
\]
Remark that the solution satisfies $\iprod{f_N(\omega^N),\omega^N}_{H} = - \nu \snorm{\omega^N}_{\dot{H}^1}^2$ for $f_N(\omega^N) = \nu \omega_{xx}^N + P_N(u_x^N\omega^N + 2u^N\omega_x^N)$.
Hence, we proceed in a similar manner to the proof of Theorem 1.4.4. in Boritchev-Kuksin~\cite{Boritchev_Kuksin_2021_onedimensional}.
Applying It\^{o}'s formula to  $G(\omega^N) = \snorm{\omega^N}_{\dot{H}^m}^2$, we have 
\begin{eqnarray}  \nonumber
\E[\snorm{\omega^N(t)}_{\dot{H}^m}^2]&=& \snorm{\omega^N(0)}_{\dot{H}^m}^2 -2\nu \int_0^t \E[\snorm{\omega^N(s)}_{\dot{H}^{m+1}}^2]ds	\\
&& + 2\int_0^t \E[\iprod{\omega^N,2\partial_x(u^N\omega^N) -  u_x^N\omega^N}_{\dot{H}^m}]ds	+ B_m^N t, 
	\label{td-Hm}
\end{eqnarray}
where $B_m^N \coloneqq \sum_{0 <\abs{k}\leq N} \abs{k}^{2m} b_k^2$. 
Multiplying the factor $e^{\nu t}$, we have again, by the It\^{o} formula, 
\begin{eqnarray} \label{td-Hm-exp}
\quad \E[\snorm{\omega^N(t)}_{\dot{H}^m}^2]&=& e^{-\nu t}\snorm{\omega^N(0)}_{\dot{H}^m}^2 
+\nu \int_0^t e^{-\nu(t-s)}\E[\snorm{\omega^N(s)}_{\dot{H}^{m}}^2]ds \\
&& -2\nu \int_0^t e^{-\nu(t-s)}\E[\snorm{\omega^N(s)}_{\dot{H}^{m+1}}^2]ds	\nonumber \\
&& + 2\int_0^t e^{-\nu(t-s)}\E[\iprod{\omega^N,2\partial_x(u^N\omega^N) - u_x^N\omega^N}_{\dot{H}^m}]ds	+ \frac{B_m^N}{\nu}(1-e^{-\nu t}). \nonumber
\end{eqnarray}
Here, another quantity will also be needed. Again, applying It\^{o}'s formula to 
$F(\omega^N) = e^{\sigma^\prime \snorm{\omega^N(t)}_{H}^2}$, we obtain 
\begin{eqnarray} \label{eq:Ito_exp}
    \E \left[e^{\sigma^\prime \snorm{\omega^N(t)}_{H}^2} \right]
    &=& e^{\sigma^\prime \snorm{\omega^N(0)}_{H}^2}
    +\E \left[\int_0^t \sigma^\prime e^{\sigma^\prime \snorm{\omega^N(s)}_{H}^2} 
    \{-2\nu\snorm{\omega^N(s)}_{\dot{H}^{1}}^2
      +B_0^N\}ds \right] \\
    && 
    + 2 (\sigma^\prime)^2 \E \left[\int_0^t e^{\sigma^\prime \snorm{\omega^N(s)}_{H}^2} \sum_{|k|\le N} b_k^2 \langle \omega(s), e_k \rangle_{H}^2 ds\right].\nonumber
\end{eqnarray}
Here, 
\begin{equation*}
2 \sigma' \sum_{k\in\mathbb{Z}^\ast}b_k^2\iprod{\omega(t),e_k}_H^2\leq 
2 \sigma' B_0 \norm{\omega(t)}_{H}^2. 
\end{equation*}
Using Assumption $3\sigma' B_0 \le \nu$, we have
{ $$
    \E \left[e^{\sigma^\prime \snorm{\omega^N(t)}_{H}^2}\right]
    \leq e^{\sigma^\prime \snorm{\omega^N(0)}_{H}^2} 
    + \E\left[\int_0^t \sigma^\prime e^{\sigma^\prime \snorm{\omega^N(s)}_{H}^2} 
    \left(-\frac{4\nu}{3}\snorm{\omega^N(s)}_{H}^2 + B_0\right) ds \right].
$$}
By the same argument as in Theorem 1.4.4 of \cite{Boritchev_Kuksin_2021_onedimensional}, applying the Gronwall lemma, we find that
\begin{equation}\label{L2_expectation_estimate}
\E\left[e^{\sigma^\prime \snorm{\omega^N(t)}_{H}^2}\right]\leq e^{-t} \E \left[e^{\sigma^\prime \norm{\omega_0}_{H}^2}\right] + C(\sigma,\nu,B_0),
\end{equation}
for any $N\in \N$.  

We now return to \eqref{td-Hm} and estimate the second term on the right-hand side of \eqref{td-Hm}. By the Cauchy-Schwarz inequality,  we have
\begin{eqnarray*}
\abs{\iprod{\partial_x\omega^N, u^N\omega^N}_{\dot{H}^m}} &\leq& \snorm{\omega^N}_{\dot{H}^{m+1}} \norm{\partial_x^m(u^N\omega^N)}_{H}, \\
\abs{\iprod{\omega^N, { u_x^N\omega^N}}_{\dot{H}^m}} &\leq& \snorm{\omega^N}_{\dot{H}^m} \norm{\partial_x^m(u^N_x\omega^N)}_{H}.
\end{eqnarray*}
H\"older's inequality 
%with $p=2m/k$, $q=2m/(m-k)$ and $r=2$ 
imply
\begin{equation*}
	\norm{\partial_x^m (u^N\omega^N) }_{H} \leq C(m) \sum_{k=0}^m \norm{\partial_x^k u^N \partial_x^{m-k}\omega^N}_{H}
	\leq C(m) \sum_{k=0}^m \norm{\partial_x^k u^N}_{L^{2m/k}} \norm{\partial_x^{m-k}\omega^N}_{L^{2m/(m-k)}}.
\end{equation*}
From the Gagliardo--Nirenberg inequality \eqref{GN_ineq} with $\beta=k$, $\gamma=2m/k$, $\alpha=m+1$ and $p=q=2$, we have
\[
\norm{\partial_x^k u^N}_{L^{2m/k}} \leq C \norm{\partial_x^{m+1} u^N}_{H}^{\theta(k)} \norm{u^N}_{H}^{1-\theta(k)} = C\norm{u^N}_{\dot{H}^{m+1}}^{\theta(k)} \norm{u^N}_{H}^{1-\theta(k)},
\]
where $\theta(k) = \frac{2mk+m-k}{2m(m+1)}$. Similarly, Gagliardo--Nirenberg inequality \eqref{GN_ineq} with $\beta=m-k$, $\gamma=2m/(m-k)$, $\alpha=m+1$ and $p=q=2$ also yields
\[
\norm{\partial_x^{m-k} u^N}_{L^{2m/(m-k)}} \leq C \norm{u^N}_{\dot{H}^{m+1}}^{\theta^\prime(k)} \norm{u^N}_{H}^{1-\theta^\prime(k)},
\]
for $\theta^\prime(k)=\frac{2m^2-2mk+k}{2m(m+1)}$. Hence, we have
\begin{equation*}
	\norm{\partial_x^k u^N}_{L^{2m/k}} \norm{\partial_x^{m-k}\omega^N}_{L^{2m/(m-k)}} \leq \norm{u^N}_{\dot{H}^{m+1}}^{\theta(k)} \norm{u^N}_{H}^{1-\theta(k)}
	\snorm{\omega^N}_{\dot{H}^{m+1}}^{\theta^\prime(k)} \snorm{\omega^{N}}_{H}^{1-\theta^\prime(k)}. 
\end{equation*}
Since $\norm{u^N}_{\dot{H}^m}  \leq\norm{u^N_x}_{\dot{H}^{m}} = \norm{\mathcal{H}(\omega^N)}_{\dot{H}^m}= \snorm{\omega^N}_{\dot{H}^m}$ for $m\geqq 0$,{   where the last equality follows from Proposition~\ref{Hilbert}}, we obtain 
\begin{equation*} 
\sum_{k=0}^m \norm{\partial_x^k u^N}_{L^{2m/k}} \norm{\partial_x^{m-k}\omega^N}_{L^{2m/(m-k)}}  \leq C(m) \snorm{\omega^N}_{\dot{H}^{m+1}}^{\theta+\theta^\prime} \snorm{\omega^N}_{H}^{2-\theta-\theta^\prime}.
\end{equation*}
Due to $\theta+\theta^\prime = 1-\frac{1}{2(m+1)}$, we have
\begin{equation*}
	\abs{\iprod{\partial_x\omega^N, u^N\omega^N}_{\dot{H}^m}} \leq C(m) \snorm{\omega^N}_{\dot{H}^{m+1}}^{2-\frac{1}{2(m+1)}}  \snorm{\omega^N}_{H}^{1+\frac{1}{2(m+1)}}.
\end{equation*}
Similarly, we obtain
\begin{equation*}
	\abs{\iprod{\omega^{N},u_{x}^{N}\omega^{N}}_{\dot{H}^{m}}}\leq C(m) \snorm{\omega^{N}}_{\dot{H}^{m+1}}^{2-\frac{1}{2(m+1)}} \snorm{\omega^{N}}_{H}^{1+\frac{1}{2(m+1)}}.
\end{equation*}
Consequently, it follows from Young's inequality 
%\blue with $p=(4m+4)/(4m+3)$ and $q=4m+4$ \black 
that
\begin{equation*}
	2\E[\iprod{\omega^N, 2\partial_x(u^N\omega^N) -  u_x^N\omega^N}_{\dot{H}^m}]	\leq \nu \E[\snorm{\omega^N}_{\dot{H}^{m+1}}^2] + C(m,\nu) \E[\snorm{\omega^N}_{H}^{4m+6}].
\end{equation*}
 Substituting this estimate into \eqref{td-Hm} yields 
\begin{equation*}
	\E[\snorm{\omega^N(t)}_{\dot{H}^m}^2]	+ \nu \int_0^t\E[\snorm{\omega^N(s)}_{\dot{H}^{m+1}}^2] ds 
    \leq \snorm{\omega^N(0)}_{\dot{H}^m}^2+ C(m,\nu) \int_0^t \E[\snorm{\omega^N(s)}_{H}^{4m+6}] ds+ B_m^N t.
\end{equation*}
Let us note that $\E[\snorm{\omega^{N}(s)}_{H}^{4m+6}] \leq C(m,\sigma)\E[e^{\sigma'\snorm{\omega^{N}(s)}_{H}^{2}}]$. 
In addition, recalling \eqref{L2_expectation_estimate},
\begin{eqnarray} \label{-nu_estimate}
&& \E[\snorm{\omega^N(t)}_{\dot{H}^m}^2]+\nu \int_0^t \E[\snorm{\omega^N(s)}_{\dot{H}^{m+1}}^2]ds \\
&&\leq \snorm{\omega^N(0)}_{\dot{H}^m}^2
   + B_m t + C(m, \nu, \sigma)\int_0^t (e^{-s} \E[e^{\sigma^\prime \norm{\omega_0}_{H}^2}] + C(\sigma,\nu,B_0)) ds. \nonumber   
\end{eqnarray} \black
Applying the same estimates to (\ref{td-Hm-exp}) 
gives rise to 
$$
    \E[\snorm{\omega^N(t)}_{\dot{H}^m}^2]\leq e^{-\nu t} \norm{\omega_0}_{\dot{H}^m}^2 +\int_0^t e^{-\nu(t-s)}\{C(m, \nu, \sigma)e^{\sigma^\prime\norm{\omega_0}_{H}^2}+ C(\sigma, \nu, B_0, B_m)]\}ds.
$$
Therefore, we obtain
\begin{equation*}
\E[\snorm{\omega^N(t)}_{\dot{H}^m}^2] \leq C(m,\nu,\sigma) (1 + \norm{\omega_0}_{\dot{H}^{m}}^2 + 	e^{\sigma^\prime \norm{\omega_0}_{H}^2}).
\end{equation*}
As $N\rightarrow \infty$, Theorem~\ref{sgclmg_global}, Proposition~\ref{Hm_as_eq} and Proposition~\ref{Hm_as_eq_5} 
\black yield
\begin{equation*}
\E[\norm{\omega(t)}_{\dot{H}^m}^2]	\leq C(m,\nu,\sigma) (1 + \norm{\omega_0}_{\dot{H}^m}^2 + e^{\sigma^\prime \norm{\omega_0}_{H}^2}).
\end{equation*}
Finally, from \eqref{-nu_estimate}, we also find
\begin{equation*}
    \E\left[\int_0^t\snorm{\omega^N(s)}_{\dot{H}^{m+1}}^2 ds\right]
    \leq \frac{1}{\nu}\norm{\omega_0}_{\dot{H}^m}^2 + tC_m(1+e^{\sigma^\prime\norm{\omega_0}_{H}^2}),
\end{equation*}
which finishes the proof.
\end{proof}
%Let \( B_b(\dot{H}^m) \) be the space of bounded Borel measurable %functions on $\dot{H}^m$. Let $P_t\;(t\ge0)$ be the Markovian semigroup %for \eqref{gCLMG_eq_r}, defined as
%\[
%    P_t\varphi(y) = \E\left[\varphi(\omega(t; y))\right], \quad  %\varphi\in B_b(\dot{H}^m),\quad y\in \dot{H}^m.
%\]
$\mathrm{Theorem}$ \ref{Stochastic_Hm_estimate} results in the existence of an invariant measure for the stochastic gCLMG equation. 

\begin{theorem}\label{thm:inv_existence}
    Let $m\in \N$, $a = -2$, $\nu>0$, and $\omega_0\in\dot{H}^m$.
    Let $m_\ast \in \N$ be such that $m_\ast >m$. Assume $B_{m_\ast}<\infty$. Then there exists an invariant measure $\mu$ on $\dot{H}^m$ for the transition semigroup $(P_t)_{t\ge 0}$ associated with \eqref{gCLMG_eq_r}. Moreover, $\mu$ is supported by $\dot{H}^{m+1}$. 
    In other words,  
    \begin{equation}\label{support}
        \int_{\dot{H}^m} \norm{y}_{\dot{H}^{m+1}}^2 \mu(dy)< \infty.
    \end{equation}
\end{theorem}
\begin{proof}
    %Let  $\mu_0\in \mathcal{P}(\dot{H}^m)$  be the distribution of %$\omega_0$, and 
    Define for $T\ge 1$, \black
    \begin{equation*}
        R^\ast_T\delta_{\omega_0} = \frac{1}{T}\int_0^T P^\ast_t\delta_{\omega_0} dt,
    \end{equation*}
    where $P^\ast_t$ denotes the adjoint of $P_t$.
    If the sequence $\left\{R^\ast_T\delta_{\omega_0}\right\}_{T\in \N}$ is tight, there exists an invariant measure by Krylov-Bogoliubov method (see Corollary 3.1.2 of \cite{da1996ergodicity}).
    Therefore, it is sufficient to show the tightness of $\{R^\ast_T\delta_{\omega_0}\}_{T\in\N}$.

For $R>0$, let $\overline{B^R_{m+1}}\coloneqq \{u\in\dot{H}^{m+1}; \norm{u}_{\dot{H}^{m+1}}\leq R\}$ and consider the complement
    \[
    (\overline{B^R_{m+1}})^c = \dot{H}^{m}\backslash \overline{B^R_{m+1}}.
    \]
By Chebyshev's inequality and \eqref{Hm_integration_estimate} with a parameter  $\sigma^{\prime}$ in Theorem 3.1, we have
\begin{align*}
    R^\ast_T\delta_{\omega_0}((\overline{B^R_{m+1}})^c)
    & = \frac{1}{T}\int_0^T P^\ast_t\delta_{\omega_0}((\overline{B^R_{m+1}})^c) dt\\
    & = \frac{1}{T}\int_0^T \mathbb{P}(\norm{\omega(t; \omega_0)}_{\dot{H}^{m+1}}>R) dt\\
    & \leq \frac{1}{R^2T}\int_0^T \E[\norm{\omega(t; \omega_0)}_{\dot{H}^{m+1}}^2]dt\\
    & \leq \frac{1}{R^2 T} \Big(\frac{1}{\nu}\norm{\omega_0}_{\dot{H}^m}^2 + CT(1+ e^{\sigma^\prime \norm{\omega_0}_{H}^2})\Big)\\
    & \leq\dfrac{C'}{\nu R^2},
\end{align*}
where the constant $C'$ depends on $\omega_0$, and is independent of $T$.
Hence, we have
\begin{equation*}
    R^\ast_T\delta_{\omega_0}(\overline{B^R_{m+1}})\ge 1-\dfrac{C}{\nu R^2}.
\end{equation*}
By the Rellich-Kondrachov theorem, $\overline{B^R_{m+1}}$ is compact in $\dot{H}^{m}$. Thus, the family of measures $\left\{R^\ast_T\delta_{\omega_0}\right\}_{T\in \N}$ is tight in $\mathcal{P}(\dot{H}^{m})$. 
Consequently, according to the Prokhorov Theorem, there exist a subsequence $\{T_j\}_{j\in \N}$ and a limit $\mu\in\mathcal{P}(\dot{H}^m)$ such that $R^\ast_{T_j} \delta_{\omega_0}$ converges weakly to $\mu$ and is an invariant measure for $P_t$. Moreover, for any $T\in\N$ and $\omega_0 \in \dot{H}^{m}$, we have
% for any $f\in B_b(\dot{H}^{m})$,  
% \begin{align*}
%     \langle R^\ast_T\delta_{\omega_0}, f \rangle
%    &=\int_{\dot{H}^{m}} \norm{y}_{\dot{H}^{m+1}}^2 R^\ast_T\mu_0(dy)\\
%    &=\frac{1}{T} \int_{\dot{H}^{m}} f \int_0^T P^\ast_t\mu_0(dy) dt\\
%    &=\frac{1}{T}\int_0^T\left\langle f, P^\ast_t\delta_{\omega_0}\right\rangle dt= \frac{1}{T}\int_0^T\left\langle P_t f,\delta_{\omega_0}\right\rangle dt.
% \end{align*}
% {\color{red}Hence, for $f=\| y \|^2_{\dot{H}^{m+1}}$, we have}
\begin{align*}
\langle R^\ast_T\delta_{\omega_0}, \|y\|^2_{\dot{H}^{m+1}}\rangle
%    &=\frac{1}{T} \int_0^T\left\langle\E\norm{\omega(t; %y)}_{\dot{H}^{m+1}}^2  ,
% \mu_0\right\rangle dt\\
% &=\frac{1}{T}\int_0^T\left\langle \|y\|^2_{\dot{H}^2_{m+1}}, P^\ast_t\delta_{\omega_0}\right\rangle dt= \frac{1}{T}\int_0^T\left\langle P_t \|y\|^2_{\dot{H}^2_{m+1}},\delta_{\omega_0}\right\rangle dt \\
    &=\frac{1}{T}\int_0^T \E \left[\norm{\omega(t; \omega_0)}_{\dot{H}^{m+1}}^2\right] dt
    %\quad(\text{by}\quad\eqref{Fubini})
    \leq C.
\end{align*}
The inequality follows from \eqref{Hm_integration_estimate}. Since $R^\ast_{T_j}\delta_{\omega_0}$ converges to $\mu$ weakly  in $\dot{H}^{m}$ \black as $j\to\infty$, 
we get $\langle \mu, \|y\|^2_{\dot{H}^{m+1}}\rangle \le C$. It follows 
from Corollary 11.1.5 of \cite{Boritchev_Kuksin_2021_onedimensional} (in the Appendix; Proposition~\ref{Hm_as_eq_5}) that 
\eqref{support} holds.
\end{proof} \black

\section{Uniqueness of invariant measures for the large viscosity case}
In this section, we check that, if $\nu$ is sufficiently large, the uniqueness of the invariant measure holds for the stochastic gCLMG equation. As in the previous section, suppose that initial data is nonrandom for simplicity. The same result can be obtained for random initial data.

%We fix parameters \(T > 0\) and \(m \in \mathbb{N}\), and assume that \%(\omega_0 \in \dot{H}^m\). Furthermore, for some \(m_\ast > m\), we assume that \(B_{m_\ast} < \infty\); in other words, \(\xi \in \dot{X}_T^{m_\ast}\).
%Then, as previously mentioned, there exists the global solution $\omega\in X_T^m$ for \eqref{gCLMG_eq_r}. In this section, we consider only nonrandom initial data for simplicity.
First, we consider the case when the initial data belongs to $\dot{H}^1$.
{\revcol In the same way as in the proof of Proposition 5.2 of \cite{Barrera_Tolle}, the next lemma is proved.\black}

{\revcol
\begin{lemma}\label{dif_estimate}
    Let $\nu >0$, $a = -2$. Let $\omega^{(1)}$ and $\omega^{(2)}$ be  solutions of \eqref{gCLMG_eq_r} for initial data $\omega_0^{(1)} \in \dot{H}^1$ and $ \omega_0^{(2)}\in\dot{H}^1$ respectively. Assume $B_{1_\ast} <+\infty$ for some $1_{\ast} >1$. If we write $\tilde\omega\coloneqq\omega^{(1)}-\omega^{(2)}$, then there exists a constant $C_\ast>0$ such that, for any $t\ge0$,
    \begin{equation}\label{exp_Hestimate}
        \norm{\tilde\omega(t)}_{H}^2
        \leq \norm{\tilde\omega_0}_{H}^2
        \exp\left\{-\nu t+\frac{C_{\ast}}{\nu}\int_0^t \snorm{\omega^{(2)}(s)}_{\dot{H}^1}^2\,ds\right\}.
    \end{equation}
\end{lemma}

\begin{proof}
By the definition of $\omega^{(1)} $and $\omega^{(2)}$, we get
\begin{align*}
\omega^{(1)}_{t}-u^{(1)}_{x}\omega^{(1)}-2u^{(1)}\omega^{(1)}_{x}-\nu\omega^{(1)}_{xx} = \xi_t, \quad u^{(1)}_{x} = \mathcal{H}\omega^{(1)}\\
\omega^{(2)}_{t}-u^{(2)}_{x}\omega^{(2)}-2u^{(2)}\omega^{(2)}_{x}-\nu\omega^{(2)}_{xx} = \xi_t, \quad u^{(2)}_{x} = \mathcal{H}\omega^{(2)}.
\end{align*} 
Write $\tilde{u} =u^{(1)}-u^{(2)}$.
By subtracting the two equations, we obtain
$$
\tilde\omega_t-\nu\tilde\omega_{xx} - u_x^{(1)}\tilde\omega-\tilde{u}_x\omega^{(2)}-2u^{(1)}\tilde\omega_x-2\tilde{u}\omega_x^{(2)}=0 .
$$
Multiplying $\tilde\omega$ to both sides and integrating yields
$$
\frac{1}{2}\dfrac{d}{dt}\|\tilde\omega\|_{H}^2+\nu \norm{ \tilde{\omega}}_{\dot{H}^{1}}^{2}
=\iprod{\tilde{u}_x\omega^{(2)}, \tilde\omega}_{H}+2\iprod{\tilde{u}\omega^{(2)}_{x}, \tilde\omega}_{H},
$$
the two terms carrying $u^{(1)}$ having vanished by $\iprod{u^{(1)}_{x}\tilde\omega+2u^{(1)}\tilde\omega_{x},\tilde\omega}_{H}=0$.
The Sobolev embedding $\dot{H}^1\subset L^\infty$ gives
$$
\abs{\iprod{\tilde{u}_x\omega^{(2)}, \tilde\omega}_{H}}
\leq \snorm{\omega^{(2)}}_{L^\infty}\norm{\mathcal{H}(\tilde\omega)}_{H}\norm{\tilde\omega}_{H}
\leq C_S\snorm{\omega^{(2)}}_{\dot{H}^1}\norm{\tilde\omega}_{H}^2,
$$
$$
2\abs{\iprod{\tilde{u}\omega^{(2)}_{x}, \tilde\omega}_{H}}
\leq 2\norm{\tilde{u}}_{L^\infty}\snorm{\omega^{(2)}}_{\dot{H}^1}\norm{\tilde\omega}_{H}
\leq 2C_S\snorm{\omega^{(2)}}_{\dot{H}^1}\norm{\tilde\omega}_{H}^2 .
$$
Hence
$$
\dfrac{d}{dt}\|\tilde\omega\|_{H}^2+2\nu \norm{ \tilde{\omega}}_{\dot{H}^{1}}^{2}
\leq 6C_S\snorm{\omega^{(2)}}_{\dot{H}^1}\norm{\tilde\omega}_{H}^2 .
$$
By the Young inequality, we get
$$
\dfrac{d}{dt}\|\tilde\omega\|_{H}^2
\leq \left[-\nu+\frac{9C_S^2}{\nu}\snorm{\omega^{(2)}}_{\dot{H}^1}^2\right]\norm{\tilde\omega}_{H}^2 ,
$$
and \eqref{exp_Hestimate} follows from Gronwall's lemma and $C_* \coloneqq 9C_S^2$.
\end{proof}
\black
}

In the same way as \cite{constantin2013ergodicity}, the next lemma is proved.  Recall $B_0 =\sum_{k\in \Z^\ast} b_k^2$.

\begin{lemma}\label{Arnaud_lemma}
    Let $ a = -2$, $\nu>0$, $\varepsilon>0$ and $\omega_0 \in \dot{H}^1$. Assume $B_{1_\ast} <+\infty$ for some $1_{\ast} >1$. If $2\varepsilon B_0\leq \nu$, the solution $\omega$ of \eqref{gCLMG_eq_r} satisfies for any $t\ge0$,
\begin{equation}\label{expexp}
\mathbb{E}\left[\exp\left(\varepsilon\left(\norm{\omega(t)}_{{H}}^2+\nu\int_0^t{\norm{\omega(s)}_{\dot{H}^1}^2}ds\right)\right)\right]
\leq\exp(\varepsilon(\norm{\omega_0}_{H}^2+tB_0)).
\end{equation}
Moreover, if $\mu$ is an invariant measure for $(P_t)_{t\ge 0}$, then
\begin{equation} \label{expbdd}
\int_{\dot{H}^1} \exp(\varepsilon\norm{y}^2_{H})\mu(dy)<\infty. 
\end{equation}
\end{lemma}

\black
\begin{proof}
    Let $Z(t)=\norm{\omega(t)}_{H}^2+\nu\int_0^t{\norm{\omega(s)}_{\dot{H}^1}^2}ds$.
    Recall that $\iprod{\omega, u_x\omega+2u\omega_x}=0$.
Then, applying It\^{o} formula, i.e., by (\ref{td-Hm}) with $m=0$, we have
$$
        dZ = -\nu\norm{\omega(t)}_{\dot{H}^1}^2dt+2\iprod{\omega(t), d\xi}+B_0 dt.
$$
We apply It\^{o} formula again to $\mathcal{E} = \exp(\varepsilon Z)$ and obtain, similarly to (\ref{eq:Ito_exp}) with $\sigma'=\varepsilon$ and $m=0$, 
\begin{align*}
\mathcal{E}(t)
&= \mathcal{E}(0)
 + \varepsilon \int_0^t \mathcal{E}(s)\Bigg[
    - \nu \norm{\omega(s)}_{\dot H^1}^2
    + B_0
    + 2\varepsilon \sum_{k\in\mathbb{Z}^\ast} b_k^2 
      \langle \omega(s), e_k\rangle_H^2
 \Bigg] \, ds \\
 &\quad + 2\varepsilon \int_0^t \mathcal{E}(s)
    \,\iprod{\omega(s), d\xi}_H .
\end{align*}

Moreover,
\begin{equation}\label{assum}
    2\varepsilon\sum_{k\in\mathbb{Z}^\ast}b_k^2\iprod{\omega(t),e_k}^2\leq \nu\norm{\omega(t)}_{H}^2\leq\nu\norm{\omega(t)}_{\dot{H}^1}^2
\end{equation}
follows from assumption $2\varepsilon B_0 \le \nu$. Thus,
%\begin{equation}
%    \mathcal{E}%(t)\leq\exp(\varepsilon\norm{\omega_0}_{H}^2)+2\varepsilon\int_0^t\mathcal{E}(s)\iprod{\omega(s), d\xi}+\varepsilon B_0\int_0^t\mathcal{E}(s)ds.
%\end{equation}
taking expectation, we have
$$
    \E[\mathcal{E}(t)]\leq\exp(\varepsilon\norm{\omega_0}_{H}^2) + \varepsilon B_0\int_0^t \E[\mathcal{E}(s)]ds.
$$
Then, equation \eqref{expexp} follows from Gronwall lemma. 

Next, we prove \eqref{expbdd}. Let $\mu$ be an invariant measure. 
We use a stationary solution $\omega(t)$ of \eqref{gCLMG_eq_r} with invariant law $\mu$ and write $Y = \norm{\omega(t)}_{H}^2$. 
Similarly, applying the It\^{o} formula to $G=\exp (\varepsilon Y)$, we obtain
\begin{align*}
    &\mathbb{E}\left[ \exp\left( \varepsilon \|\omega(t)\|_{H}^2 \right) \right]
    + 2\varepsilon\nu \mathbb{E}\left[\int_0^t 
        \|\omega(s)\|_{\dot{H}^1}^2 \exp\left(\varepsilon \|\omega(s)\|_{H}^2 \right)ds \right]\notag \\
    &= 2\varepsilon^2 \mathbb{E}\left[\int_0^t \sum_{k \in \mathbb{Z}^\ast}
        b_k^2 \langle \omega(s), e_k \rangle^2
        \exp\left(\varepsilon \|\omega(s)\|_{H}^2 \right) ds \right] \notag \\
    &\quad + \varepsilon B_0 \mathbb{E}\left[ \int_0^t 
        \exp\left( \varepsilon \|\omega(s)\|_{H}^2 \right) ds \right]
    + \mathbb{E}\left[ \exp\left( \varepsilon \|\omega_0\|_{H}^2 \right) \right].
\end{align*}
Using \eqref{assum} again, we have
\begin{align}\label{ito_estimate}
 &\E\left[\exp(\varepsilon\norm{\omega(t)}_{H}^2)\right]
 +\varepsilon\nu\E\left[\int_0^t\norm{\omega(s)}_{\dot{H}_1}^2
   \exp(\varepsilon\norm{\omega(s)}_{H}^2)ds\right] \notag\\
 &\leq \varepsilon B_0\E\left[\int_0^t
   \exp(\varepsilon\norm{\omega(s)}_{H}^2)ds\right]
 + \E\left[\exp(\varepsilon\norm{\omega_0}_{H}^2)\right].
\end{align}
Since $\omega$ is stationary with law $\mu$, we have
$$
    \E\left[\exp(\varepsilon\norm{\omega(t)}_{H}^2)\right]=\E\left[\exp(\varepsilon\norm{\omega_0}_{H}^2)\right]=\int_{\dot{H}^1}\exp{(\varepsilon\norm{y}_{H}^2)\mu(dy)},
$$
and
\begin{align*}
    \E\left[\int_0^t\norm{\omega(s)}_{\dot{H}^1}^2\exp(\varepsilon\norm{\omega(s)}_{H}^2)ds\right]
    &= \int_0^t \int_{\dot{H}^1}\norm{y}_{\dot{H}^1}^2\exp(\varepsilon\norm{y}_{H}^2)\mu(dy) ds\\
    &= t\int_{\dot{H}^1}\norm{y}_{\dot{H}^1}^2\exp(\varepsilon\norm{y}_{H}^2)\mu(dy).
\end{align*}
    We deduce from \eqref{ito_estimate} that
    \begin{align}\label{H^1_reduce}
        \int_{\dot{H}^1}\norm{y}_{\dot{H}^1}^2\exp(\varepsilon\norm{y}_{H}^2)\mu(dy)\leq \frac{B_0}{\nu}\int_{\dot{H}^1}\exp{(\varepsilon\norm{y}_{H}^2)\mu(dy)}.
    \end{align}
    Let $R>0$, then by \eqref{H^1_reduce}, 
    \begin{align*}
        \int_{\dot{H}^1}\exp(\varepsilon\norm{y}_{H}^2)\mu(dy)
        &=\int_{\norm{y}_{H}\leq R}\exp(\varepsilon\norm{y}_{H}^2)\mu(dy)+\int_{\norm{y}_{H}>R}\exp(\varepsilon\norm{y}_{H}^2)\mu(dy)\\
        &\leq \exp(\varepsilon R^2)+\frac{1}{R^2}\int_{\dot{H}^1}\norm{y}_{\dot{H}^1}^2\exp(\varepsilon\norm{y}_{H}^2)\mu(dy)\\
        &\leq \exp(\varepsilon R^2)+ \frac{B_0}{\nu R^2}\int_{\dot{H}^1}\exp{(\varepsilon\norm{y}_{H}^2)\mu(dy)}.
    \end{align*}
    Since $R$ is arbitrary, taking $\dfrac{B_0}{\nu R^2}=\dfrac{1}{2}$, we get \eqref{expbdd}.
\end{proof}

{\revcol
\begin{prop}\label{expmixing} Let $\nu>0$, $a=-2$. Let $\omega^{(1)}$ and $\omega^{(2)}$ be  solutions of \eqref{gCLMG_eq_r} for initial data $\omega_0^{(1)} \in \dot{H}^1$ and $ \omega_0^{(2)}\in\dot{H}^1$ respectively. Assume $B_{1_\ast} <+\infty$ for some $1_{\ast} >1$. 
    Define $\tilde\omega\coloneqq\omega^{(1)}-\omega^{(2)}$ and let $C_{\ast}$ be as in Lemma \ref{dif_estimate}. If $\nu^3\ge2C_{\ast}B_0$, then there exists a constant 
    $\gamma_{\ast} = \gamma_{\ast}(\nu, B_0)>0$ and the following inequality holds for any $t\ge0$,
\begin{equation}\label{eq:expmixing}
    \mathbb{E}\left[\norm{\tilde\omega(t)}_{H}^2\right]\leq
    \norm{\tilde\omega_0}_{H}^2\,e^{-\gamma_{\ast}t}\,
    \exp\left(\frac{C_{\ast}}{\nu^2}\snorm{\omega_0^{(2)}}_{H}^2\right).
\end{equation}
\end{prop}

\begin{proof}
    Taking $\varepsilon=\frac{C_{\ast}}{\nu^2}$ in Lemma \ref{Arnaud_lemma}, we obtain
    \begin{equation*}
        \E\left[\exp\left(\frac{C_{\ast}}{\nu}\int_0^t\snorm{\omega^{(2)}(s)}_{\dot{H}^1}^2 ds\right)\right]
        \leq
        \exp\left(\frac{C_{\ast}}{\nu^2}\snorm{\omega_0^{(2)}}_{H}^2+\frac{C_{\ast}B_0}{\nu^2}t\right).
    \end{equation*}
Taking the expectation in \eqref{exp_Hestimate} of Lemma \ref{dif_estimate} and using this bound, we have
 \begin{equation*}
     \E\left[\norm{\tilde\omega(t)}_{H}^2\right]
    \leq \norm{\tilde\omega_0}_{H}^2\,e^{-\nu t}
    \exp\left(\frac{C_{\ast}}{\nu^2}\snorm{\omega_0^{(2)}}_{H}^2+\frac{C_{\ast}B_0}{\nu^2}t\right).
 \end{equation*}
 Setting \begin{equation}\label{eq:gammastar}
    \gamma_{\ast}\coloneqq\nu-\frac{C_{\ast}B_0}{\nu^2}>0,
\end{equation}
we obtain \eqref{eq:expmixing}.
\end{proof}
\black
}

\begin{theorem}\label{thm:uniqueness}
Let $m\in \N$, $\nu>0$, and $a = -2$. Let $m_{\ast} \in \N$ be such that $m_{\ast}>m$. Assume $B_{m_{\ast}} <+\infty$.
If  {\revcol $\nu^3\ge 4 C_{\ast} B_0$}, with $C_{\ast}$ as in Lemma \ref{dif_estimate}, the invariant measure $\mu\in\mathcal{P}(\dot{H}^m)$ for the transition semigroup $(P_t)_{t\ge 0}$ associated with \eqref{gCLMG_eq_r} is unique.
\end{theorem}
\begin{proof}
Let $\mu$ and $\tilde\mu$ be the invariant measures for the transition semigroup $(P_t)_{t\ge 0}$ associated with \eqref{gCLMG_eq_r}.
{ Let $\omega(t;x)$ and $\omega(t;y)$ be the solutions of \eqref{gCLMG_eq_r} with the initial datum $x,y$ in $\dot{H}^m$. 
Write $\tilde\omega(t) \coloneqq\omega(t;x)-\omega(t;y)$.}
We then consider the Lipschitz-dual distance  (Kantorovich--Rubinstein distance or Wasserstein-1 distance)  between them. For a Lipschitz function $\varphi$ on $\dot{H}^m$ , we have
\begin{eqnarray} \label{ineq:dif}
        \left|\iprod{\varphi,\mu}-\iprod{\varphi,\tilde\mu}\right|
        &=& \left|\int_{\dot{H}^m}\int_{\dot{H}^m}(P_t\varphi(x)-P_t\varphi(y))\mu(dx)\tilde\mu(dy)\right|\\  \nonumber
        &\leq& \norm{\varphi}_{\mathrm{Lip}}\int_{\dot{H}^m}\int_{\dot{H}^m}{\mathbb{E}\left[\|\omega(t;x)-\omega(t;y)\|_{\dot{H}^m}\right] \black}{\mu}(dx)\tilde\mu(dy). 
\end{eqnarray} 
Note that 
    \begin{equation}\label{Holder}
        \mathbb{E}\left[\norm{\tilde\omega(t)}_{\dot{H}^m}\right]
        \leq \left(\mathbb{E}\left[\norm{\tilde\omega(t)}_{H}^2\right]\right)^{\frac{1}{2(m+1)}}
        \left(\mathbb{E}\left[\norm{\tilde\omega(t)}_{\dot{H}^{m+1}}^2\right]\right)^{\frac{m}{2(m+1)}},
    \end{equation} 
which is proved by the interpolation inequality of \eqref{Sobolev_interpolation} in the Appendix, and by the H\"{o}lder inequality in $d\mathbb{P}$. Using this inequality
\eqref{Holder},  
\begin{align*}
    &\int_{\dot{H}^m} \int_{\dot{H}^m}
    \mathbb{E}\left[ \|\omega(t;x) - \omega(t;y)\|_{\dot{H}^m} \right] \,
    \mu\left(dx\right) \tilde\mu\left(dy\right) \notag \\
    &\leq \left\{
    \int_{\dot{H}^m} \int_{\dot{H}^m}
    \mathbb{E}\left[ \|\tilde\omega(t)\|_{H}^2 \right] \mu\left(d x\right) \tilde\mu\left(dy\right) \right\}^{\frac{1}{2(m+1)}}\notag\\
    &\qquad\times\left\{
    \int_{\dot{H}^m} \int_{\dot{H}^m}
    \mathbb{E}\left[ \|\tilde\omega(t)\|_{\dot{H}^{m+1}}^2 \right] \mu\left(d x\right) \tilde\mu\left(dy\right) \right\}^{\frac{m}{2(m+1)}} \notag \\
\end{align*}
Here, we may estimate using Proposition \ref{expmixing},
{\revcol
\begin{eqnarray*}
&& \left\{\int_{\dot{H}^m} \int_{\dot{H}^m}
    \mathbb{E}\left[ \|\tilde\omega(t)\|_{H}^2 \right] 
    \mu\left(d x\right) \tilde\mu\left(dy\right) \right\}^{\frac{1}{2(m+1)}} \\  
    &\le&  \exp \left(-\frac{\gamma_{\ast}t}{2(m+1)}\right)
    \left\{\int_{\dot{H}^m} \int_{\dot{H}^m}
    \|x-y\|_{H}^2\,e^{\frac{C_{\ast}}{\nu^2}\|y\|_{H}^2}
    \mu\left(dx\right) \tilde\mu\left(dy\right)\right\}^{\frac{1}{2(m+1)}},
\end{eqnarray*}
where, by $\|x-y\|_{H}^2\le2\|x\|_{H}^2+2\|y\|_{H}^2$ and $\|y\|_{H}^2\le\frac{\nu^2}{C_{\ast}}e^{\frac{C_{\ast}}{\nu^2}\|y\|_{H}^2}$,
\begin{eqnarray*}
&&\int_{\dot{H}^m} \int_{\dot{H}^m}\|x-y\|_{H}^2\,e^{\frac{C_{\ast}}{\nu^2}\|y\|_{H}^2}\mu(dx)\tilde\mu(dy)\\
&\le& 2\int_{\dot{H}^m}\|x\|_{H}^2\mu(dx)\int_{\dot{H}^m}e^{\frac{C_{\ast}}{\nu^2}\|y\|_{H}^2}\tilde\mu(dy)
+\frac{2\nu^2}{C_{\ast}}\int_{\dot{H}^m}e^{\frac{2C_{\ast}}{\nu^2}\|y\|_{H}^2}\tilde\mu(dy).
\end{eqnarray*}
The right-hand side is finite owing to \eqref{expbdd}: the first integral because $\norm{x}_{H}^2\le\frac{\nu^2}{C_{\ast}}e^{\frac{C_{\ast}}{\nu^2}\norm{x}_{H}^2}$, and the others directly, \eqref{expbdd} being applied with $\varepsilon=\frac{C_{\ast}}{\nu^2}$ and with $\varepsilon=\frac{2C_{\ast}}{\nu^2}$, the latter being permitted since $\nu^3\ge4C_{\ast}B_0$. Thus this quantity converges to zero as $t \to +\infty$.
\black
}

%Noting that \eqref{support} indicates $\mu(\dot{H}^m/\dot{H}^{m+1})=0$ for any invariant measure $\mu$, we can change the domain of the integration from $\dot{H}^m$ to $\dot{H}^{m+1}$. Hence, retaking the constant appearing in \eqref{Hm_expectation_estimate} if necessary, by young inequality and H\"{o}lder inequality, we have
Next, by (\ref{Hm_expectation_estimate}) with  {\revcol $\sigma'=\frac{C_{\ast}}{\nu^2}$\black}, 
\begin{eqnarray*}
&& \int_{\dot{H}^m} \int_{\dot{H}^m}
    \mathbb{E}\left[ \|\tilde\omega(t)\|_{\dot{H}^{m+1}}^2 \right] 
    \mu\left(d x\right) \tilde\mu\left(dy\right) \\
& {\revcol \lesssim} &
\int_{\dot{H}^m} \int_{\dot{H}^m}
\left(1+ \|x\|^2_{\dot{H}^{m+1}}
+ \|y\|^2_{\dot{H}^{m+1}}\right.
\left.
+ {\revcol e^{\frac{C_{\ast}}{\nu^2}\|x\|_{H}^2}}
+ {\revcol e^{\frac{C_{\ast}}{\nu^2}\|y\|_{H}^2}}\right) \,
\mu\left(dx\right) \tilde\mu\left(dy\right),  
\end{eqnarray*}
which is again finite due to Lemma \ref{Arnaud_lemma} and \eqref{support}. 
{\revcol Consequently, the quantity (\ref{ineq:dif}) converges to zero as $t \to +\infty$, and $\mu=\tilde\mu$.\black} \end{proof}

%\begin{align}
%    &\int_{\dot{H}^m} \int_{\dot{H}^m} 
%    \mathbb{E}\left[ \|\omega^{(1)} - \omega^{(2)}\|_{\dot{H}^m}^2 \right] \,
%    \mu\left(d\omega_0^{(1)}\right) \tilde\mu\left(d\omega_0^{(2)}\right) \\
%    &\leq 
%    C_{m+1}^\prime 
%    \exp\left( \frac{t}{m+1} \left( -\frac{\nu}{4} + \frac{C_1 B_0}{\nu^2} \right) \right)
%    \int_{\dot{H}^m} \int_{\dot{H}^m}
%    \|\tilde\omega_0\|_{H}^{\frac{1}{m+1}} \,
%    \exp\left( \frac{1}{m+1} \cdot \frac{C_1}{\nu^2} \left( %\|\omega_0^{(1)}\|_{H}^2 + \|\omega_0^{(2)}\|_{H}^2 \right) \right) %\notag \\
%    &\quad 
%\end{align}
%Recalling \eqref{support} and \eqref{expbdd}, we show that the last integral converges. Therefore, if $\nu^3\ge C_0B_0$ for some $C_0>0$, the right-hand side converges to zero as $t\to\infty$.

{\revcol
The argument below follows the proof of Lemma 6.1 and Theorem 6.2 of \cite{Barrera_Tolle}. Before drawing the consequence of Proposition \ref{expmixing}, we recall two notions. For $\mu_1,\mu_2\in\mathcal{P}(H)$ with finite second moments, the Wasserstein-2 distance between them is
\[
\mathcal{W}_2(\mu_1,\mu_2)\coloneqq\Bigl(\inf_{\pi}\int_{H\times H}\norm{y-z}_{H}^2\,\pi(dy,dz)\Bigr)^{\frac12},
\]
the infimum being taken over all $\pi\in\mathcal{P}(H\times H)$ whose marginals are $\mu_1$ and $\mu_2$. If $\mu$ is the unique invariant measure of $(P_t)_{t\ge0}$, the $\varepsilon$-mixing time from $\omega_0\in \dot{H}^1$ is
\[
\tau^{\omega_0}_{\mathrm{mix}}(\varepsilon)\coloneqq\inf\bigl\{t\ge0 : \mathcal{W}_2(P^\ast_t\delta_{\omega_0},\mu)\le\varepsilon\bigr\},\qquad \varepsilon>0 .
\]
\black
}

\begin{corollary}\label{cor:mixing}Let $\nu>0$, $a = -2$ and $\omega_0\in \dot{H}^1.$ Assume $B_{1_{\ast}}<+\infty$ for some $1_{\ast}>1$.
{\revcol Let $C_{\ast}$ and $\gamma_\ast$ be as in Lemma \ref{dif_estimate} and \eqref{eq:gammastar}, and assume $\nu^3\ge4C_{\ast}B_0$.
Let $\mu\in\mathcal{P}(\dot{H}^1)$ be the unique invariant measure. Then the solution of \eqref{gCLMG_eq_r} is exponentially mixing in the Wasserstein-2 distance, and for any $\varepsilon>0$,
\begin{equation}\label{eq:mixingtime}
\tau^{\omega_0}_{\mathrm{mix}}(\varepsilon)\leq\frac{2}{\gamma_{\ast}}
\left[\frac12\log\left\{\left(2\norm{\omega_0}_{H}^2+\frac{2\nu^2}{C_{\ast}}\right)M\right\}
+\log\frac1\varepsilon\right],
\end{equation}
where 
\[
    M\coloneqq\int_{\dot{H}^1}\exp\Bigl(\frac{2C_{\ast}}{\nu^2}\norm{y}_{H}^2\Bigr)\mu(dy) < \infty.
\]}
\end{corollary}
{\revcol
\begin{proof}
Since $\nu^3\ge4C_{\ast}B_0$, the bound \eqref{expbdd} of Lemma \ref{Arnaud_lemma} applied with $\varepsilon=\frac{2C_{\ast}}{\nu^2}$
 gives $M<\infty$. Let $y$ be an $\dot{H}^1$-valued, $\mathcal{F}_0$-measurable random variable with law $\mu$, and let $\omega(\,\cdot\,;\omega_0)$ and $\omega(\,\cdot\,;y)$ be the solutions of \eqref{gCLMG_eq_r}
   with initial data $\omega_0$ and $y$. Let $\pi_t\in\mathcal{P}(H\times H)$ be the law of the pair $(\omega(t;\omega_0),\omega(t;y))$.
    Its first marginal is $P^\ast_t\delta_{\omega_0}$ and its second one is $P^\ast_t\mu=\mu$.
     Since $\mu(\dot{H}^1)=1$, it follows from the definition of $\mathcal{W}_2$ and \eqref{eq:expmixing} that
\[
\begin{split}
\mathcal{W}_2\bigl(P^\ast_t\delta_{\omega_0},\mu\bigr)^2
\leq\int_{H}\E\bigl[\norm{\omega(t;\omega_0)-\omega(t;y)}_{H}^2\bigr]\mu(dy)
&=\int_{\dot{H}^1}\E\bigl[\norm{\omega(t;\omega_0)-\omega(t;y)}_{H}^2\bigr]\mu(dy)\\
&\leq e^{-\gamma_{\ast}t}\int_{\dot{H}^1}\norm{\omega_0-y}_{H}^2e^{\frac{C_{\ast}}{\nu^2}\norm{y}_{H}^2}\mu(dy).
\end{split}
\]
Then, we have 
\[
\begin{split}
\int_{\dot{H}^1}\norm{\omega_0-y}_{H}^2e^{\frac{C_{\ast}}{\nu^2}\norm{y}_{H}^2}\mu(dy)
&\leq2\norm{\omega_0}_{H}^2\int_{\dot{H}^1}e^{\frac{C_{\ast}}{\nu^2}\norm{y}_{H}^2}\mu(dy)
+\frac{2\nu^2}{C_{\ast}}\int_{\dot{H}^1}e^{\frac{2C_{\ast}}{\nu^2}\norm{y}_{H}^2}\mu(dy)\\
&\leq\left(2\norm{\omega_0}_{H}^2+\frac{2\nu^2}{C_{\ast}}\right)M .
\end{split}
\]
Taking square roots gives, for any $t\ge0$,
\[
\mathcal{W}_2\bigl(P^\ast_t\delta_{\omega_0},\mu\bigr)
\leq e^{-\frac{\gamma_{\ast}t}{2}}
\left\{\left(2\norm{\omega_0}_{H}^2+\frac{2\nu^2}{C_{\ast}}\right)M\right\}^{\frac12},
\]
so that the solution of \eqref{gCLMG_eq_r} is exponentially mixing, and \eqref{eq:mixingtime} follows from the definition of the $\varepsilon$-mixing time.
\end{proof}
\black
}

{\revcol
\begin{remark}
In fact, on $H$, \cite{Barrera_Tolle} yields the uniqueness of the invariant measure and the exponential mixing, under a large-viscosity condition of the same form as ours which is recorded in Appendix C. Since uniqueness on $H$ implies uniqueness among the invariant measures carried by $\dot{H}^m$, the conclusions of Theorem \ref{thm:uniqueness} and of Corollary \ref{cor:mixing} are direct consequences of that work as well. On $\dot{H}^m$ with $m\ge1$, the existence of an invariant measure for all viscosity and its regularity, that is, Theorem \ref{thm:inv_existence}, are not consequences of \cite{Barrera_Tolle}. Appendix C shows that the coercivity hypotheses fail when the equation is regarded on $\dot{H}^m$ instead, because the cancellation $\iprod{u_x\omega+2u\omega_x,\omega}_{H}=0$ is specific to the $L^2$ inner product.
\end{remark}
}

\renewcommand{\thesection}{A}\setcounter{theorem}{0}
\section*{Appendix A. Known auxiliary results}
We provide some known results that are used in this paper. The following result follows from the definition of the Hilbert transform and the norm of $\dot{H}^m$.

\begin{prop}\label{Hilbert}
	For $m\ge 0$, $m\in \N$ \black and $f \in \dot{H}^{m}(\mathbb{S}^{1})$, we have $\norm{\mathcal{H}(f)}_{\dot{H}^{m}(\mathbb{S}^{1})} = \norm{f}_{\dot{H}^{m}(\mathbb{S}^{1})}$.
\end{prop}

Basically, the Hilbert transform is defined by an integral operator with singular kernel.
Some important properties such as $L^{p}(1<p<\infty)$ boundedness are found in a text book \cite{king_2009} for example.
%\begin{prop}[\cite{DaPrato_Zabczyk_2014}; Proposition 3.16]\label{Kolmogrov} \blue The norm is unclear \black
%    Let X be Gaussian on Hilbert space H. Assume that $\E(X(t))=0, t\ge 0$, and that there exist $M>0$ and $\gamma\in (0,1]$ such that
%    \[
%    \E\left[\norm{X(t)-X(s)}^2\right]\leq M(t-s)^\gamma, \quad \forall  t, s\ge 0.
%    \]
%    Then $X$ has an $\alpha$-H\"{o}lder continuous version, for any $\alpha\in (0, \frac{\gamma}{2})$.
%\end{prop}
The following three inequalities are used in this paper. See~\cite{Boritchev_Kuksin_2021_onedimensional} for the proof.
\begin{prop} 
For $\sigma>1/2$ and $f,g \in \dot{H}^{\sigma}(\mathbb{S}^{1})$, 	
\begin{equation} \label{Hm_product}
\norm{fg}_{\dot{H}^\sigma(\mathbb{S}^{1})} \leq \norm{f}_{\dot{H}^\sigma(\mathbb{S}^{1})} \norm{g}_{\dot{H}^\sigma(\mathbb{S}^{1})}.
\end{equation}
\end{prop}
\begin{prop}[Sobolev interpolation] 
Let $\sigma=(1-\theta)\sigma_1+\theta \sigma_2$ for $0\leq \sigma_{1}\leq \sigma_{2}$ and $\theta \in [0,1]$. Then $f \in \dot{H}^{\sigma_{1}}(\mathbb{S}^{1}) \cap \dot{H}^{\sigma_{2}}(\mathbb{S}^{1}) $ implies $f \in \dot{H}^{\sigma}(\mathbb{S}^{1})$.
In addition, the following inequality holds.
\begin{equation}\label{Sobolev_interpolation}
\norm{f}_{\dot{H}^{\sigma}(\mathbb{S}^{1})} \leq \norm{f}_{\dot{H}^{\sigma_{1}}(\mathbb{S}^{1})}^{1-\theta} \norm{f}_{\dot{H}^{\sigma_{2}}(\mathbb{S}^{1})}^{\theta}
\end{equation}
\end{prop}
\begin{prop}[Gagliardo-Nirenberg inequality; \cite{Boritchev_Kuksin_2021_onedimensional}, Lemma 1.3.3 or \cite{bahouri2011fourier}, Theorem 2.44]\label{Gagliardo_Nirenberg_ineq}
Let $\alpha \black>\beta\ge 0$ and $1\leq p,q, \gamma \leq \infty$.
Then there exists $C(\beta,p,q, \gamma,\alpha \black)>0$ such that 
\begin{equation}
	\norm{\partial^{\beta}h}_{L^{ \gamma \black}(\mathbb{S}^{1})} \leq C\norm{\partial^{\alpha \black}h}_{L^{p}(\mathbb{S}^{1})}^{\theta} \norm{h}_{L^{q}(\mathbb{S}^{1})}^{1-\theta},
	\label{GN_ineq}
\end{equation}
holds for any $\theta \in [\frac{\beta}{\alpha },1)$ satisfying $\beta - \frac{1}{\gamma \black} 
= \theta( \alpha -\frac{1}{p}) - (1-\theta)\frac{1}{q}$.
\end{prop}

% Let$X_{0}$, $X$, $X_1$ be Banach spaces with $X_{0}\subset X \subset X_{1}$ whose embeddings are injective and continuous. Let also $X_0$ and $X_1$ are reflexive, and the embedding $X_0 \rightarrow X_1$  is compact. For any $T>0$ and $\alpha_0$, $\alpha_1>1$, we introduce the function space 
% \[
% 	\mathcal{Y} = \mathcal{Y}(0,T;\alpha_{0},\alpha_{1};X_{0},X_{1})\coloneqq 	\{v \in L^{\alpha_{0}}([0,T];X_{0}) \mid v^\prime = \frac{dv}{dt} \in L^{\alpha_{1}}([0,T];X_{1})\}
% 	\]
% endowed with the norm;
% \[
% \norm{v}_{\mathcal{Y}} \coloneqq \norm{v}_{L^{\alpha_{0}}([0,T];X_{0})} + \norm{v'}_{L^{\alpha_{1}}([0,T];X_{1})}
% \]
% Then $\mathcal{Y}$ is Banach space and the following holds~\cite{Temam_NS_1984}.
% \begin{theorem}[\cite{Temam_NS_1984}, Theorem 2.1]\label{bochner_cpt_embedding}
% The embedding $\mathcal{Y} \to L^{\alpha_{0}}([0,T];X)$ is compact.
% \end{theorem}

% For any $0\leq m_1 \leq m_2$ let us define the space
% \[
% Z \coloneqq 	\{u \in L^{2}([0,T];\dot{H}^{m_{2}}) \mid u' = \frac{du}{dt} \in L^{2}([0,T];\dot{H}^{m_{1}})\} 
% \]
% endowed with the Hilbert norm
% \[
% \norm{u}_{Z}^{2} \coloneqq \norm{u}_{L^{2}([0,T];\dot{H}^{m_{2}})}^{2}+ \norm{u'}_{L^{2}([0,T];\dot{H}^{m_{1}})}^{2}.
% \]
% Then the spaced $Z$ becomes a Hilbert space and satisfies the following property.
% \begin{theorem}[\cite{Boritchev_Kuksin_2021_onedimensional}, Theorem 11.2.6]\label{Thm_Bochner_XTHm_embedding}
% The space $Z$ is embedded in $C([0,T];\dot{H}^{(m_{1}+m_{2})/2})$
% \end{theorem}

The following two propositions have been shown in  Boritchev and Kuksin~\cite{Boritchev_Kuksin_2021_onedimensional}, which are used in the proof of Theorem~\ref{Stochastic_Hm_estimate} and (\ref{support}).

\begin{prop}[\cite{Boritchev_Kuksin_2021_onedimensional}; Theorem 11.1.4]\label{Hm_as_eq}
Let $m_{1} \geqq m_{2}$ and $\{\xi_{N}\}_{N \in \N}$ be a sequence of random variables in $\dot{H}^{m_1}$.
Suppose that there exist $K>0$ and $1\leq \alpha < \infty$ such that $E[\norm{\xi_{N}}^{\alpha}_{\dot{H}^{m_1}}]\leq K$ is satisfied  for any $N \in \N$.
Assume also that $\xi_N$ converges to a random variable $\xi$ a.s in the space of $\dot{H}^{m_2}$.
Then $\xi = \xi^\prime$ a.s., where $\xi^\prime$ is a random variable in $\dot{H}^{m_1}$ with $\E[\norm{\xi^\prime}_{\dot{H}^{m_1}}^{\alpha}] \leq K$.
\end{prop}

\begin{prop}[\cite{Boritchev_Kuksin_2021_onedimensional}; Corollary 11.1.5]\label{Hm_as_eq_5} 
Let $m_1 \ge m_2$ and $\{\mu_n\}_{n}$ be a sequence of measures in $\mathcal{P}(\dot{H}^{m_1})$ such that $\langle \|u\|^p_{\dot{H}^{m_1}}, \mu_n \rangle \le K$ for all $n$ and some $1 \le p \le +\infty$. 
Assume also $\mu_n \to \mu$ in $\mathcal{P}(\dot{H}^{m_2}).$ 
Then, $\mu(\dot{H}^{m_1})=1,$ so $\mu$ may be regarded as a measure on $\dot{H}^{m_1},$ and $\langle \|u\|^p_{\dot{H}^{m_1}}, \mu \rangle \le K.$ 
\end{prop}

%\begin{prop}[\cite{Simon1974}; hyper-contractibility] \label{nelson}
%Let $\xi=\{\xi_n\}_n$ be a sequence of standard independent and identically %distributed real Gaussian random variables. Let $k\in \N$ and let $\%{P_j(\xi)\}_{j\in \N}$ be a sequence of polynomials of degree at most k. Then for $p \ge 2$, the follwing holds.
%$$\E [|\sum_{j\in \N} P_j(\xi)|^p]^{\frac{1}{p}} \le (p-1)^{\frac{k}{2}} %\E[|\sum_{j\in \N} P_j(\xi)|^2]^{\frac12}.$$ 
%\end{prop}
\renewcommand{\thesection}{B}\setcounter{theorem}{0}
{\revcol
\section*{Appendix B. Proof of Lemma \ref{S_heat_wellposedness}}
The detailed verification of the moment estimate used in the proof of Lemma \ref{S_heat_wellposedness} is collected here.

\begin{proof}[Detailed proof of Lemma \ref{S_heat_wellposedness}]
We check the Kolmogorov criterion (see \cite{DaPrato_Zabczyk_2014}).
Without loss of generality, we may assume $m_\ast\in (m, m+2]$. It is known that \eqref{S_heat_eq} has a unique solution written as \eqref{S_conv}, see \cite{DaPrato_Zabczyk_2014}. For $t, s\in[0,T]$ with $s<t$, we have
\begin{align}
     v(t,x)-v(s,x) 
    &= \sum_{k\in \mathbb{Z}^*} \int_s^t e^{-\nu(t-\tau)\Aop}b_k e_kd\beta_k(\tau) + 
    \sum_{k\in \mathbb{Z}^*} \int_0^s \{e^{-\nu(t-\tau)\Aop}-e^{-\nu(s-\tau)\Aop}\}b_k e_kd\beta_k(\tau).\label{eq:heat-sol}
\end{align}
For arbitrarily $\gamma\in [0,1]$, the expectation of the first term is given by
    \begin{align*}
    \E\left[\norm{\sum_{k\in \mathbb{Z}^*} \int_s^t e^{-\nu(t-\tau)\Aop}b_k e_kd\beta_k(\tau)}_{\dot{H}^m}^2\right]
        &=\sum_{k\in \mathbb{Z}^*} |k|^{2m}\int_s^t e^{-2\nu(t-\tau)k^2}|b_k|^2d\tau\\    
        &=\sum_{k\in \mathbb{Z}^*} |k|^{2m}\frac{|b_k|^2}{2\nu k^2}(1-e^{-2\nu(t-s)k^2})\\
        &\leq \sum_{k\in \mathbb{Z}^*} |k|^{2m_\ast}\frac{|b_k|^2}{2\nu k^2}2^\gamma(\nu k^2)^\gamma (t-s)^\gamma\\
        &= \frac{B_{m_\ast}}{(2\nu )^{1-\gamma}} (t-s)^\gamma.
    \end{align*}
    The expectation of the second term becomes
    \begin{align*}
        &\E\left[\norm{\sum_{k\in \mathbb{Z}^*} \int_0^s \{e^{-\nu(t-\tau)\Aop}-e^{-\nu(s-\tau)\Aop}\}b_k e_kd\beta_k(\tau)}_{\dot{H}^m}^2\right] \\
        &=\sum_{k\in\mathbb{Z}^*}|k|^{2m}\int_0^s\{e^{-\nu(t-\tau)k^2}-e^{-\nu(s-\tau)k^2}\}^2b_k^2d\tau\\
        &=\sum_{k\in\mathbb{Z}^*}|k|^{2m}\int_0^s\{1-e^{-\nu(t-s)k^2}\}^2e^{-2\nu(s-\tau)k^2}b_k^2d\tau\\
        &\leq C({\nu, m, m_*, \delta})(t-s)^{2\delta}\\
        &\quad\times\sum_{k\in\mathbb{Z}^*}\int_0^s(s-\tau)^{\frac{m_*-m}{2}-1-\delta}\left[(2\nu k^2(s-\tau))^{\frac{2+m-m_*}{2}+\delta}e^{-2\nu(s-\tau)k^2}\right]|k|^{2m_*}b_k^2d\tau.
        \end{align*}
        Note that $\frac{2+m-m_*}{2}+\delta\ge 0$ and $\sup_{r\ge 0} r^a e^{-r} < \infty$ for $a\ge 0$. Then, taking a sufficiently small $\delta\in[0,1]$ satisfying $\delta<\frac{m_\ast-m}{2}$, we obtain
        \begin{equation*}
        \E\left[\norm{\sum_{k\in \mathbb{Z}^*} \int_0^s \{e^{-\nu(t-\tau)\Aop}-e^{-\nu(s-\tau)\Aop}\}b_k e_kd\beta_k(\tau)}_{\dot{H}^m}^2\right]\leq C({\nu, m, m_*, \delta})(t-s)^{2\delta}B_{m_*}T^{\frac{m_*-m}{2}-\delta}.
    \end{equation*}
    Since $\gamma\in [0, 1]$ is arbitrary, we take $\gamma = 2\delta$. Therefore, taking the expectation of (\ref{eq:heat-sol}), we have that for all $t, s>0$, 
   $$
        \E[\norm{v(t)-v(s)}_{\dot{H}^m}^2]\leq C(t-s)^\gamma.
   $$
Since $v(t)$ is a Gaussian process, the Kolmogorov criterion (Proposition 3.16 of \cite{DaPrato_Zabczyk_2014}) is satisfied.  
%by Theorem \ref{Kolmogrov} as shown in the appendix. 
Therefore, we conclude $v\in X_T^m$ almost surely. More precisely, by the choice of $\gamma$ and $\delta$, the Kolmogorov continuity theorem implies that $v$ is $\alpha$-H\"{o}lder continuous for any $\alpha \in \bigl(0, \min\bigl\{\frac12, \frac{m_\ast-m}{2}\bigr\}\bigr)$.
\end{proof}
\black
}

{\revcol
\renewcommand{\thesection}{C}\setcounter{theorem}{0}
\section*{Appendix C. Comparison with the framework of \cite{Barrera_Tolle}}

This appendix contains the verification of Hypotheses \textup{(A)}--\textup{(K)} of \cite[\S3]{Barrera_Tolle} for the Gelfand triple
\begin{equation}\label{eq:triple}
\dot{H}^1\subset H\equiv H^*\subset\dot{H}^{-1}
\end{equation}
and the drift
\begin{equation}\label{eq:driftBT}
\mathcal{A}(\omega)\coloneqq\nu\omega_{xx}+\mathcal{N}(\omega),\qquad
\mathcal{N}(\omega)\coloneqq u_{x}\omega+2u\omega_{x},
\end{equation}
of \eqref{gCLMG_eq_r} with $a=-2$, together with the noise $\xi$.

Our mild solutions of Definition \ref{def:solution} are solutions in the sense of \cite[Definition 3.5]{Barrera_Tolle}, by the equivalence between mild and analytically weak solutions for semilinear equations with additive noise \cite[Chapter 5]{DaPrato_Zabczyk_2014}, so that the transition semigroup $(P_t)_{t\ge0}$ of Section 3 is the restriction to $\dot{H}^m$, for every integer $m\ge1$, of the semigroup on $H$ of that work.

The hypotheses involve constants $\alpha,\hat\alpha\ge2$, $\beta,\hat\beta\ge0$, $\delta_1,\hat\delta_1,\delta_2,\delta_3,K_3>0$, $K_1,K_2,C_4\in\R$, $C_2\ge0$ and measurable functions $\eta,\rho,\rho_{\mathrm E}:\dot{H}^1\to[0,\infty)$. Written for the triple \eqref{eq:triple} and the drift \eqref{eq:driftBT}, they read, for $\omega,\varphi,\psi,\omega^{(1)},\omega^{(2)}\in\dot{H}^1$ and with $\tilde\omega=\omega^{(1)}-\omega^{(2)}$:
\begin{align*}
&\text{\textup{(A)}}\quad \R\ni\lambda\mapsto\iprod{\mathcal{A}(\omega+\lambda\varphi),\psi}\ \text{is continuous};\\
&\text{\textup{(B)}}\quad 2\iprod{\mathcal{A}(\omega),\omega}\le-\delta_1\snorm{\omega}_{\dot{H}^1}^{\alpha};\\
&\text{\textup{(C)}}\quad 2\iprod{\mathcal{A}(\omega),\omega}\le K_1\bigl(1+\norm{\omega}_{H}^2\bigr)-\hat\delta_1\snorm{\omega}_{\dot{H}^1}^{\hat\alpha};\\
&\text{\textup{(D)}}\quad 2\iprod{\mathcal{A}(\omega^{(1)})-\mathcal{A}(\omega^{(2)}),\tilde\omega}\le\bigl[-\delta_2+\eta(\omega^{(1)})+\rho(\omega^{(2)})\bigr]\norm{\tilde\omega}_{H}^2,\\
&\qquad\qquad 0\le\eta+\rho\le C_2\snorm{\cdot}_{\dot{H}^1}^{\alpha}\norm{\cdot}_{H}^{\beta};\\
&\text{\textup{(E)}}\quad \textup{(D)}\ \text{with }-\delta_2\ \text{replaced by }K_2\ \text{and with }\hat\alpha,\hat\beta,\rho_{\mathrm E}\ \text{in place of }\alpha,\beta,\rho,\\
&\qquad\qquad \text{where }\rho_{\mathrm E}\ \text{is required to be hemicontinuous when }\eta\equiv0;\\
&\text{\textup{(F)}}\quad \norm{\mathcal{A}(\omega)}_{\dot{H}^{-1}}^{\hat\alpha/(\hat\alpha-1)}\le K_3\bigl(1+\snorm{\omega}_{\dot{H}^1}^{\hat\alpha}\bigr)\bigl(1+\norm{\omega}_{H}^{\hat\beta}\bigr);\\
&\text{\textup{(G)}}\quad 2\iprod{\mathcal{A}(\omega),\omega}\le C_4-\delta_3\norm{\mathcal{A}(\omega)}_{\dot{H}^{-1}};\\
&\text{\textup{(H)}}\quad B_0<\infty;\\
&\text{\textup{(I)}}\quad \mathbb{P}\Bigl(\sup_{t\in[0,T]}\snorm{\xi(t)}_{\dot{H}^1}\le\delta\Bigr)>0\quad\text{for all }T,\delta>0;\\
&\text{\textup{(J)}}\quad \frac{2B_0C_2}{\lambda_0\delta_1}\le\delta_2
\quad\text{and}\quad B_0\le\tfrac12\lambda_3\delta_1,\ \text{with }C_2\ \text{as in \textup{(D)}},\\
&\qquad\qquad\text{for some }\lambda_0,\lambda_3\in(0,1)\text{ with }\lambda_0+\lambda_3=1;\\
&\text{\textup{(K)}}\quad \textup{(J)}\ \text{with the first inequality replaced by}\\
&\qquad\qquad \frac{2B_0C_2}{\lambda_0\delta_1}\le\delta_2-\gamma
\quad\text{for some }\gamma\in(0,\delta_2] .
\end{align*}
Under \textup{(J)} the invariant measure of \cite[Theorem 2.1]{Barrera_Tolle} is unique, and under \textup{(K)} the convergence to it is exponentially fast in the Wasserstein-2 distance by \cite[Theorem 2.2]{Barrera_Tolle}.

Let $a=-2$, $\nu>0$ and $B_{1_\ast}<\infty$ for some $1_\ast>1$, and let $C_S$ be the constant of the embedding $\dot{H}^1(\mathbb{S}^1)\subset L^\infty(\mathbb{S}^1)$, so that $C_\ast=9C_S^2$. Then Hypotheses \textup{(A)}--\textup{(I)} hold with the following choices:
\begin{center}
\begin{tabular}{ll}
\textup{(B)}, \textup{(C)} & $\alpha=\hat\alpha=2$, $\delta_1=\hat\delta_1=2\nu$, $K_1=0$;\\
\textup{(D)} & $\alpha=2$, $\beta=0$, $\delta_2=\nu$, $\eta\equiv0$, $\rho(\omega^{(2)})=\frac{C_{\ast}}{\nu}\snorm{\omega^{(2)}}_{\dot{H}^1}^2$, $C_2=\frac{C_{\ast}}{\nu}$;\\
\textup{(E)} & $\hat\alpha=\hat\beta=2$, $K_2=\frac{9\sqrt2}{2}+\frac4\nu-\nu$, $\eta\equiv0$, $\rho_{\mathrm E}(\omega^{(2)})=\bigl(\frac{3\sqrt2}{2}+\frac4\nu\bigr)\snorm{\omega^{(2)}}_{\dot{H}^1}^2\snorm{\omega^{(2)}}_{H}^2$;\\
\textup{(F)}, \textup{(G)} & $K_3=2\nu^2+50C_S^2$; \ $\delta_3=\frac{4\nu}{\nu+10C_S}$, $C_4=\frac{\delta_3\nu}{2}$.
\end{tabular}
\end{center}
Moreover, with $\lambda_0=\frac34$ and $\lambda_3=\frac14$, Hypothesis \textup{(J)} is
\begin{equation}\label{eq:nu-condition}
\nu^3\ge12C_S^2B_0\qquad\text{and}\qquad\nu\ge4B_0 ,
\end{equation}
so that under \eqref{eq:nu-condition} the equation falls within the framework of \cite{Barrera_Tolle} at the level of $H$. If the first inequality in \eqref{eq:nu-condition} is strict, then Hypothesis \textup{(K)} holds as well. The noise may vanish: for $\xi\equiv0$, Theorem \ref{thm:uniqueness} and Corollary \ref{cor:mixing} hold for every $\nu>0$, and the unique invariant measure is the Dirac measure $\delta_0$ at the only stationary solution $\omega\equiv0$ of the deterministic equation; see \cite[Remark 4.4]{Barrera_Tolle}.

Hypothesis \textup{(D)} with these constants is the differential inequality established in the proof of Lemma \ref{dif_estimate}, of which \eqref{exp_Hestimate} is the integrated form. The remaining hypotheses are verified by elementary computations, and we have only recorded their constants.

Finally we show that \cite{Barrera_Tolle} does not apply beyond $H$, which is why Section 4 is not written in that framework: on the Gelfand triples
\[
\dot{H}^{m+1}\subset\dot{H}^m\subset\dot{H}^{m-1},\qquad m\ge1,
\]
the coercivity Hypotheses \textup{(B)} and \textup{(C)} fail.

\begin{lemma}\label{lem:exact-identity}
Let $\varphi\coloneqq\cos x+\sin2x$. Then, for every integer $m\ge0$,
\begin{equation}\label{eq:exact-identity}
\iprod{\mathcal{N}(\varphi),\varphi}_{\dot{H}^m}=\tfrac32\,\pi\,(4^m-1).
\end{equation}
\end{lemma}
\begin{proof}
From $\mathcal{H}(\cos kx)=\sin kx$ and $\mathcal{H}(\sin kx)=-\cos kx$ for $k\ge1$, we have
\[
u_x\varphi=\tfrac12\sin2x-\cos3x-\tfrac12\sin4x,\qquad
2u\varphi_x=\sin2x-\tfrac52\cos3x-\tfrac32\cos x-\sin4x,
\]
whence
\[
\mathcal{N}(\varphi)=\tfrac32\sin2x-\tfrac72\cos3x-\tfrac32\cos x-\tfrac32\sin4x .
\]
For $a,b\in\R$ and $k,\ell\ge1$, we have
\[
\iprod{a\cos kx,b\cos \ell x}_{\dot{H}^m}
=\iprod{a\sin kx,b\sin \ell x}_{\dot{H}^m}
=\pi k^{2m}ab\,\delta_{k\ell},\qquad
\iprod{a\cos kx,b\sin \ell x}_{\dot{H}^m}=0 ,
\]
where $\delta_{k\ell}$ is the Kronecker delta. Hence
\[
\iprod{\mathcal{N}(\varphi),\varphi}_{\dot{H}^m}
=\pi\Bigl[1^{2m}\cdot\bigl(-\tfrac32\bigr)\cdot1+2^{2m}\cdot\tfrac32\cdot1\Bigr]
=\tfrac32\pi(4^m-1) . \qedhere
\]
\end{proof}

\begin{prop}\label{prop:BTfail}
Let $m\ge1$. There exist no constants $K_1\in\R$, $\hat\delta_1>0$ and $\hat\alpha\ge2$ such that Hypothesis \textup{(C)},
\[
2\iprod{\mathcal{A}(\omega),\omega}_{\dot{H}^m}\le K_1\bigl(1+\norm{\omega}_{\dot{H}^m}^2\bigr)-\hat\delta_1\norm{\omega}_{\dot{H}^{m+1}}^{\hat\alpha},
\]
holds for all $\omega\in\dot{H}^{m+1}$; in particular Hypothesis \textup{(B)}, the case $K_1=0$, fails as well. 
\end{prop}
\begin{proof}
Take $\omega_\lambda\coloneqq\lambda\varphi$ with $\lambda>0$, where $\varphi$ is as in Lemma \ref{lem:exact-identity}. Since $\mathcal{N}$ is homogeneous of degree two, \eqref{eq:exact-identity} gives
\[
2\iprod{\mathcal{A}(\omega_\lambda),\omega_\lambda}_{\dot{H}^m}
=-2\nu\lambda^2\norm{\varphi}_{\dot{H}^{m+1}}^2+3\pi(4^m-1)\lambda^3
\longrightarrow+\infty\qquad(\lambda\to\infty).
\]
The right-hand side of Hypothesis \textup{(C)} is bounded above by $K_1(1+\lambda^2\norm{\varphi}_{\dot{H}^m}^2)$, which grows only quadratically in $\lambda$. This is a contradiction.
\end{proof}
}

\section*{Acknowledgment}
This research was supported by JSPS KAKENHI 19KK0066 and 23H00086.
The authors  are very grateful to Yuta Tsuji, from whom we learned a lot through our discussions.
The authors are grateful to the referee for the careful reading of the manuscript and for the valuable comments. {\revcol In particular, the referee drew our attention to the work of Barrera and T\"olle \cite{Barrera_Tolle}, which is the starting point of the comparison carried out at the end of Section 4 and in Appendix C. That comparison led us to the estimate of Lemma \ref{dif_estimate}, and thereby to the bound \eqref{eq:mixingtime} for the mixing time, which improve on those we had obtained in the first version of this work.}

\end{document}